\documentclass{amsart}
\usepackage[utf8]{inputenc}

\usepackage{amssymb,amsmath, amsthm, amsfonts}
\usepackage{mathrsfs}
\usepackage[ruled,linesnumbered]{algorithm2e}

\usepackage[margin=1in]{geometry}
\usepackage{graphicx, subcaption}
\usepackage{xcolor}
\usepackage{tikz-cd}
\usepackage[doi=false,isbn=false,url=false,eprint=false]{biblatex}
\def\spann{\mathop{\rm span}\nolimits}

\DeclareMathOperator*{\supp}{supp}

\def\im{\mathop{\rm im}}
\def\conv{\mathop{\rm conv}}
\def\cconv{\mathop{\overline{\rm conv}}}
\def\cone{\mathop{\rm cone}}
\def\ccone{\mathop{\overline{\rm cone}}}

\newcommand{\bbZ}{\mathbb{Z}}

\newcommand{\bbR}{\mathbb{R}}

\newcommand{\bbP}{\mathbb{RP}}
\newcommand{\bbS}{\mathbb{S}}

\newcommand{\cD}{\mathcal{D}}

\newcommand{\cP}{\mathcal{P}}

\newcommand{\cV}{\mathcal{V}}

\usepackage{thmtools, thm-restate}
\declaretheorem[numberwithin=section]{theorem}

\newtheorem{corollary}[theorem]{Corollary}
\newtheorem{lemma}[theorem]{Lemma}
\newtheorem{proposition}[theorem]{Proposition}

\theoremstyle{definition}
\newtheorem{example}[theorem]{Example}

\newtheorem{remark}[theorem]{Remark}

\usepackage{hyperref}
\hypersetup{
    colorlinks=true,
    linkcolor = blue,
    citecolor = blue,
    urlcolor = blue
}

\usepackage{spectralsequences} 
\usepackage{tikz}
\usetikzlibrary{decorations.pathreplacing,matrix}
\tikzset{ 
table/.style={
  matrix of math nodes,
  row sep=-\pgflinewidth,
  column sep= 100\pgflinewidth,
  nodes={rectangle,text width=3em,align=center},
  text depth=1.25ex,
  text height=2.5ex,
  nodes in empty cells,
  ampersand replacement=\&
}
}

\usepackage{pdfpages}

\title{A Topological Approach to Simple Descriptions of Convex Hulls of Sets Defined by Three Quadrics}

\author{Grigoriy Blekherman and Alex Dunbar}

\begin{document}

\begin{abstract}
    We study the convex hull of a set $S\subset \bbR^n$ defined by three quadratic inequalities. A simple way of generating inequalities valid on $S$ is to take nonnegative linear combinations of the defining inequalities of $S$. We call such inequalities \emph{aggregations}. We introduce a new technique relating aggregations to properties of the spectral curve, i.e. the curve defined by the vanishing of the determinant polynomial, and utilizing known spectral sequences  \cite{agrachev_systems_2012}. We find new families beyond those identified in \cite{blekherman_aggregations_2022,dey_obtaining_2022}, where the convex hull is defined by aggregations. We also prove a characterization of the emptiness of the projective variety defined by $3$ homogeneous quadratics in terms of the spectral curve generalizing results of \cite{AgrachevHomologyIntersections}.
\end{abstract}
\maketitle

\section{Introduction}

Optimization of linear functions on a subset $S$ of $\mathbb{R}^n$ leads naturally to considering the convex hull of $S$. A semialgebraic set defined by quadratic inequalities is often nonconvex. One strategy for computing the convex hull of such a set is to study \emph{aggregations} of the defining inequalities \cite{blekherman_aggregations_2022,dey_obtaining_2022,yildiran_convex_2009}. We focus on the case of three quadratic inequalities. Given three linearly independent symmetric matrices $Q_1,Q_2,Q_3 \in \bbR^{(n+1)\times (n+1)}$ defining the three quadratics $f_i(x) = \begin{bmatrix}x^\top & 1 \end{bmatrix} Q_i \begin{bmatrix} x^\top & 1\end{bmatrix}^\top$, an aggregation is simply a nonnegative combination of the defining inequalities: $f_\lambda = \sum_{i = 1}^3 \lambda_i f_i$ with $\lambda_i \geq 0$. Note that $Q_\lambda = \sum_{i = 1}^{3}\lambda_i Q_i$ is a symmetric matrix which represents the quadratic function $f_\lambda$.  

The use of aggregations for obtaining the convex hull of a set is a natural idea in optimization. In linear programming, the Farkas Lemma ensures that the convex hull of a set defined by linear inequalities can be described by aggregations of these inequalities. Similarly, aggregations have been used in the contexts of integer programming \cite{bodur2018aggregation} and mixed-integer nonlinear programming problems \cite{gleixner2020generalized}. 

Let $S = \{x \in \bbR^n \; \vert \; f_i(x) \leq 0, \; i \in [3]\}$. One application of the study of aggregations is to compute a representation of the closed convex hull of $S$ as $\cconv(S) = \bigcap_{\lambda \in \Lambda_1}\{x \in \bbR^n \; \vert \; f_\lambda(x) \leq 0\}$ for some appropriate subset $\Lambda_1$ of the nonnegative orthant $\bbR^3_+$. In order to relate aggregations to convexity, we restrict the signature of the matrices $Q_\lambda$. An aggregation $\lambda$ is \emph{permissible} if $Q_\lambda$ has at most one negative eigenvalue. If $\lambda$ is a permissible aggregation, then the set $\{x \in \bbR^n \; \vert \; f_\lambda(x) < 0\}$ is either empty, convex, or the union of two connected components, each of which is convex. A permissible aggregation $\lambda$ is a \emph{good} aggregation if $\cconv(S) \subseteq \{x \in \bbR^n \; \vert \; f_\lambda(x) \leq 0\}$. In order to describe $\cconv(S)$ as an intersection of permissible aggregations, it is necessary to further restrict our attention to good aggregations. In \cite{dey_obtaining_2022}, the authors prove such a description is possible for sets defined by strict inequalities (i.e. $f_i(x) < 0$ for $i \in [3]$) under the assumption that there is a positive definite linear combination of the defining quadratics. When such a linear combination exists, we say that \emph{PDLC} holds. In \cite{blekherman_aggregations_2022}, the authors show that for a set defined by (possibly more than three, strict) quadratic inequalities, the convex hull can be computed via aggregations under a hidden convexity assumption. 

We study sets of aggregations of quadratic inequalities via the \emph{spectral curve}, the curve in $\bbP^2$ defined as the vanishing set of the determinant $g(\lambda) = \det(\lambda_1Q_1 + \lambda_2Q_2 + \lambda_3Q_3)$, using tools derived from algebraic topology. A line of work \cite{agrachev_systems_2012,AgrachevHomologyIntersections,degtyarev_number_2012,lerario_convex_2012} relates the topology of sets defined by quadratic inequalities with the topology of the sets of linear combinations of the $Q_i$ with specified number of positive eigenvalues. In this setting, the homology groups of the solution sets to systems of quadratic inequalities are computed from the cohomology groups of combinations of the $Q_i$ with specified numbers of positive eigenvalues using a spectral sequence. Note that the number of positive eigenvalues is constant on connected components of the complement of the spectral curve.

When PDLC holds, the projective variety $\cV_{\bbR}(f_1^h,f_2^h,f_3^h)\subseteq \bbP^n$ defined by the homogenizations of the $f_i$ is empty. We begin by investigating the structure of the spectral curve for the case of three quadratics which define an empty projective variety but do not necessarily satisfy PDLC. Our first result shows that when $n\geq 4$ and the projective variety $\cV_{\bbR}(f_1^h,f_2^h,f_3^h)$ is empty, $g$ is \emph{hyperbolic} when it is smooth. Recall that a homogeneous polynomial $h\in \bbR[x_1,\ldots,x_{n+1}]$ is hyperbolic with respect to a point $e$ if $h(te + a) \in \bbR[t]$ is real rooted for any point $a \in \bbR^{n+1}$. If $h$ is a homogeneous polynomial in $n+1$ variables which is hyperbolic with respect to a point $e$, then the connected component of $e$ in $\bbR^{n+1}\setminus \cV_{\bbR}(h)$ is a convex cone \cite{garding_inequality_1959}. If $h$ is hyperbolic and defines a smooth plane curve $C$, then $C$ consists of $\lfloor\frac{\deg h}{2} \rfloor$ nested ovals (and a pseudo-line if $\deg h$ is odd). Theorem \ref{thm:intro_thm1} is a generalization of the sufficient condition that there is a positive definite linear combination $Q_\mu$ of the $Q_i$ since in this case, $g$ is hyperbolic with respect to $\mu$. 

\begin{restatable}{theorem}{IntroTheoremOne}\label{thm:intro_thm1}
    Let $Q_1,Q_2,Q_3 \in \bbR^{(n+1)\times (n+1)}$ be three linearly independent symmetric matrices and $f_1^h,f_2^h,f_3^h$ the associated quadratic forms $f_i^h(x) = x^\top Q_i x$. Suppose that the spectral curve $\cV_{\bbR}(g)$ is smooth and $n \not = 2$. Then, the real projective variety $\cV_\bbR(f_1^h,f_2^h,f_3^h)$ is empty if and only if one of the following holds: 
    \begin{enumerate}
        \item[(i)] The polynomial $g$ is hyperbolic, and there is $\mu \in \bbR^3$ such that $Q_\mu$ has $n$ positive eigenvalues
        \item[(ii)] $n = 3$ and the spectral curve $\cV_\bbR(g)$ is empty.
    \end{enumerate}
    
    When $n = 2$, the spectral curve is not smooth if the variety $\cV_{\bbR}(f_1^h,f_2^h,f_3^h)$ is nonempty.
\end{restatable}

The $n = 3$ statement is \cite[Theorem 7.8]{plaumann_quartic_2011} and the statement that $g$ is hyperbolic when $n \geq 4$ and the spectral curve is smooth is \cite[Corollary 2]{AgrachevHomologyIntersections}. However, we were unable to find a full statement of Theorem \ref{thm:intro_thm1} in the literature. Theorem \ref{thm:intro_thm1} is proved in Section \ref{sec:Spectral_comps}.   

We observe that condition (i) in Theorem \ref{thm:intro_thm1} gives the possible signatures (number of positive eigenvalues minus number of negative eigenvalues) of the matrices on the interior of the nested ovals of the spectral curve. On the exterior of all ovals, the matrices must have signature 0 if $n+1$ is even and $\pm 1$ if $n+1$ is odd \cite{Vinnikov1993SelfAdjoint}. Crossing from the exterior to the interior of an oval makes the signature change by 2 so that matrices on the interior of the second innermost oval (the oval of depth $\lfloor \frac{n+1}{2} \rfloor - 1$) and the exterior of the innermost convex oval have signature $\pm (n - 1)$. Therefore, there is a hyperbolicity cone of $g$ consisting of matrices which are either positive definite or have $n-1$ positive and two negative eigenvalues. Theorem \ref{thm:intro_thm1} also has an interpretation in terms of the celebrated Helton-Vinnikov Theorem \cite{HeltonVinnikovLMIsets}, which shows that every hyperbolic plane curve admits a definite determinantal representation (i.e. the defining quadratics can be chosen to satisfy PDLC). In this case, the associated projective variety is empty. Theorem \ref{thm:intro_thm1} limits the other possible representations of hyperbolic plane curves which correspond to empty varieties. Explicitly, the only representations that correspond to an empty variety are definite representations and representations for which $n$ positive eigenvalues are attained and there is a hyperbolicity cone for which the matrices have $n-1$ positive and two negative eigenvalues. Moreover, such a representation certifies the emptiness of the variety defined by the quadratics. 

In the case that the real projective variety $\cV_{\bbR}(f_1^h,f_2^h,f_3^h)$ is empty and $n \not = 2$, the spectral sequence approach relates certificates of emptiness for the solution sets of systems of inequalities coming from aggregations to the (non)intersections of polyhedral cones with a hyperbolicity cone of $g$. In particular, computations of cohomology groups reduce to computations in convex geometry. 

Our primary application of such certificates of emptiness is a new sufficient condition for the existence of a representation of $\cconv(S)$ in terms of good aggregations. In \cite{blekherman_aggregations_2022,dey_obtaining_2022}, hidden convexity properties of the quadratic maps are leveraged to generate aggregations $\lambda$ such that restriction  $Q_\lambda \vert_H$ of $Q_\lambda$ to some hyperplane $H \subseteq \bbR^{n+1}$ is positive semidefinite. When applied to hyperplanes $H$ arising as the homogenizations of hyperplanes separating a point $y \in \bbR^n$ from $\conv(\mathrm{int}(S))$, this provides a good aggregation which certifies that $y \not \in \conv(\mathrm{int}(S))$. Utilizing a spectral spectral sequence given in \cite{agrachev_systems_2012} to compute homology groups of hyperplane sections of the solution set of a system of quadratic inequalities, we employ a similar strategy and generalize \cite[Theorem 2.7]{dey_obtaining_2022} under the additional assumption that $S$ \emph{has no points at infinity}, that is $\left\{x \in \bbR^n \; \middle \vert \; \begin{bmatrix} x^\top & 0 \end{bmatrix} Q_i \begin{bmatrix} x & 0 \end{bmatrix}^\top \leq 0\right\} = \{0\}$.  

\begin{restatable}{theorem}{ClosedConvexSufficient}\label{thm:ClosedConvexSufficient}
    Assume that $\mathrm{int}(S) \not = \emptyset$, and that $\{ x \in \bbR^{n+1} \; \vert \; x^\top Q_i x \leq 0,\, i \in [3],\, x_{n+1} = 0\} = \{0\}$. Suppose that $\cV_{\bbR}(f_1^h,f_2^h,f_3^h)$ is empty and the spectral curve is smooth and nonempty.  

    \begin{itemize}
        \item[(i)] If $n\geq 3$, then the spectral curve $g$ is hyperbolic. Let $\cP\subset \bbR^3$ be the hyperbolicity cone of $g$ such that $\mathrm{int} \cP$ consists of either positive semidefinite matrices $Q_\lambda$, or matrices $Q_\lambda$ with exactly two negative eigenvalues. If no non-zero aggregation lies in $\cP$ then there are $k\leq 6$ good aggregations $\lambda^{(1)},\lambda^{(2)}, \ldots, \lambda^{(k)}$,  such that $\cconv(S) = \{x \in \bbR^n \; \vert \; f_{\lambda^{(i)}}(x) \leq 0 , \, i \in [k]\}.$ 

        \item[(ii)] If $n=2$ and $g$ is hyperbolic, then (i) still applies. If $g$ is not hyperbolic, then there is a (possibly infinite) subset $\Lambda_1 \subseteq \bbR^n_+$ of good aggregations such that $\cconv(S) = \{x \in \bbR^n \; \vert \; f_{\lambda}(x) \leq 0, \, \lambda \in \Lambda_1\}$. 
    \end{itemize}

\end{restatable}

Theorem \ref{thm:ClosedConvexSufficient} relaxes the PDLC condition from \cite{dey_obtaining_2022} but imposes the additional restrictions that $S$ has no points at infinity and that the spectral curve is smooth. However, the most complete results in the literature apply to the setting of strict inequalities, while Theorem \ref{thm:ClosedConvexSufficient} applies to the (technically more intricate) setting of non-strict inequalities. Theorem \ref{thm:ClosedConvexSufficient} is proved in Section \ref{sec:SufficientCondition}.

Our second application of certificates of emptiness is the development of a strategy to prove the redundancy of an aggregation in a set, and in particular reduce sets of aggregations to finite subsets. Specifically, given aggregations $\lambda^{(1)},\lambda^{(2)}, \ldots, \lambda^{(k+1)}$, we interpret the condition $\bigcap_{i = 1}^{k+1} S_{\lambda^{(i)}} = \bigcap_{i = 1}^{k}S_{\lambda^{(i)}}$ in terms of a system of equations and inequalities. An application of this strategy shows that if the variety $\cV_{\bbR}(f_1^h,f_2^h,f_3^h)$ is empty and the spectral curve is smooth and hyperbolic, then the set obtained by intersecting all permissible aggregations can be recovered by intersecting only finitely many permissible aggregations. 

\begin{restatable}{theorem}{FinitePermissibleAggs}\label{thm:FinitePermissibleAggs}
Suppose that $\cV_{\bbR}(f_1^h,f_2^h,f_3^h)$ is empty and that the spectral curve is smooth and hyperbolic. Then, there is a finite subset $\Lambda_1$ of permissible aggregations such that 

\[\left\{x \in \bbR^n \; \vert \; f_\lambda(x) \leq 0 \text{ for all } \lambda \in \Lambda_1\right\} = \left\{x \in \bbR^n \; \vert \; f_\lambda(x) \leq 0 \text{ for all permissible } \lambda \in \bbR^3_+\right\}.\]
\end{restatable}

Theorem \ref{thm:FinitePermissibleAggs} is proved in Section \ref{sec:AggsFinite}.

Our final consideration is the topology of the set of good aggregations as a subset of all permissible aggregations. Specifically, in Section \ref{sec:Connected_Aggs}, we consider the connectedness properties of the set of good aggregations under PDLC and the regularity assumption that $S$ has \emph{no low dimensional components}, that is $S = \mathrm{cl}(\mathrm{int}(S))$. This leads us to a bound on the number of good aggregations needed to describe $\cconv(S)$ with these assumptions.   

\begin{restatable}{theorem}{PDLCBound}\label{thm:PDLCBound}
If $Q_1,Q_2,Q_3$ satisfy PDLC, $S = \mathrm{cl}(\mathrm{int}(S))$, $S$ has no points at infinity, and $\mathrm{int}(S) \not = \emptyset$, then there is a subset $\{\lambda^{(1)},\lambda^{(2)}, \ldots, \lambda^{(r)}\}$ of good aggregations with $r \leq 4$ such that \[\cconv(S) = \bigcap_{i = 1}^r \left\{x \in \bbR^n \; \vert \; f_{\lambda^{(i)}}(x) \leq 0\right\}.\]
\end{restatable}

Note that Theorem \ref{thm:PDLCBound} is related to \cite[Conjecture 3.2]{blekherman_aggregations_2022}, in which the authors conjecture the existence of a set $T$ defined by the intersection of three \emph{strict} quadratic inequalities satisfying PDLC such that the convex hull cannot be described using fewer than six good aggregations. The result of Theorem \ref{thm:PDLCBound} is that if the set defined by non-strict inequalities is sufficiently regular, then four aggregations suffice. Theorem \ref{thm:PDLCBound} is proved in Section \ref{sec:Agg_top}.

The remainder of the paper is organized as follows. In Section \ref{sec:notation_preliminaries}, we fix notation and review the relevant background from real algebraic geometry and algebraic topology. A detailed discussion of our main results and examples are presented in Section \ref{sec:Main_Results}. Section \ref{sec:Conclusions} presents conclusions. Proofs of the main results are in subsequent sections.

\section{Notation and Preliminaries}\label{sec:notation_preliminaries} 

We adopt the following notational conventions. For a positive integer $k$, the set $\{1,2,\ldots, k\}$ is denoted $[k]$. Given a subset $W \subseteq \bbR^n$, we let $\conv(W),\cone(W),\mathrm{int}(W),\mathrm{cl}(W),$ and $\partial W$ be the convex hull, conical hull, interior, closure, and boundary of $W$, respectively. We denote by $\cconv(W)$ and $\ccone(W)$ the closures of the convex hull and conical hull of $W$, respectively. The real projective $n$-space is denoted $\bbP^n$, the $n$-sphere is denoted $\bbS^n$, the closed unit $n$-ball is denoted $B^n$, and the nonnegative orthant of $\bbR^n$ is denoted $\bbR^n_+$. If $\lambda \in \bbR^n$ is a vector, then $\lambda_i$ denotes the $i^{th}$ component of $\lambda$ and $\supp(\lambda) = \{i \in [n] \; \vert \; \lambda_i \not = 0\}$ is the support of $\lambda$. The standard basis vectors of $\bbR^n$ are denoted by $e_i$, where $(e_i)_j = \begin{cases} 1 & i = j\\ 0 & \text{otherwise} \end{cases}$.

Throughout, we will consider the setting with three fixed symmetric matrices $Q_1,Q_2,Q_3 \in \bbR^{(n+1)\times (n+1)}$. For each $i \in [3]$, we have the block structure $Q_i = \begin{bmatrix} A_i & b_i\\ b_i^\top & c_i \end{bmatrix}$, where $A_i \in \bbR^{n\times n}, b_i \in \bbR^{n},$ and $c_i \in \bbR$. To these three matrices we associate the following:

\begin{itemize}
\item The quadratic functions $f_i(x) = \begin{bmatrix}x^\top & 1 \end{bmatrix} Q_i \begin{bmatrix} x^\top & 1\end{bmatrix}^\top$, $i \in [3]$.
\item The set $S = \{x \in \bbR^n \;\vert \; f_i(x) \leq 0, i\in [3]\}$
\item The homogenized functions $f_i^h(x,x_{n+1}) = \begin{bmatrix}x^\top & x_{n+1} \end{bmatrix} Q_i \begin{bmatrix} x^\top & x_{n+1}\end{bmatrix}^\top$, $i \in [3]$.
\item The homogenized set $S^h = \{(x,x_{n+1}) \in \bbR^{n+1} \;\vert \; f_i(x) \leq 0, i\in [3]\}$.
\item The determinant polynomial $g(x,y,z) = \det(xQ_1 + y Q_2 + zQ_3)$ and the spectral curve $\cV_\bbR(g) \subseteq \bbP^2$.
\end{itemize}

Moreover, given $\lambda \in \bbR^3$, we consider linear combinations:

\begin{itemize}
\item $Q_\lambda = \sum_{i=1}^3 \lambda_iQ_i$. The blocks $A_\lambda, b_\lambda,$ and $c_\lambda$ are defined analogously. 
\item $f_\lambda(x) = \sum_{i = 1}^{3}\lambda_i f_i(x)$ and the homogenization $f_\lambda^h = \sum_{i = 1}^{3}\lambda_i f_i^h(x)$.
\item $S_\lambda = \{x \in \bbR^n \;\vert \; f_\lambda(x) \leq 0\}$ and its homogenization $S_\lambda^h = \{(x,x_{n+1}) \in \bbR^{n+1} \;\vert \; f_\lambda^h(x) \leq 0\}$.
\end{itemize}

We set 
\[\Lambda = \{\lambda \in \bbR^3_+ \; \vert \; Q_\lambda \text{ has exactly one negative eigenvalue}\}.\]
Note that $\lambda$ is a permissible aggregation if and only if $\lambda \in \Lambda$ or $Q_\lambda \succeq 0$. Moreover, if $Q_\lambda$ has exactly one negative eigenvalue, then $\text{int}(S_\lambda)$ consists of either one convex component or two connected components, each of which is convex \cite[Lemma 5.2]{blekherman_aggregations_2022}. It is possible that $\lambda \in \Lambda$ and $\mathrm{int}(S_\lambda)$ has two connected components, each of which intersects $\mathrm{int}(S)$ nontrivially. In this case, $\conv(S) \not \subseteq S_\lambda$. So, we additionally define the set of \emph{good aggregations} to be the $\lambda \in \Lambda$ such that  $\conv(S) \subseteq S_\lambda$. 

Let $K \subseteq \bbR^{m}$ be a polyhedral cone. We denote the polar dual by $K^\circ = \{x \in \bbR^m \; \vert \; x^\top y \leq 0 \text{ for all } y \in K\}$. For a homogeneous quadratic map $f^h = (f_1^h,f_2^h,\ldots, f_m^h) : \bbR^{n+1} \to \bbR^m$, let 

\[X(K,f^h) = \{[x] \in \bbP^{n} \; \vert \; f^h([x]) \in K\}.\] 
In particular, given vectors $\lambda^{(1)}, \lambda^{(2)}, \ldots, \lambda^{(m)} \in \bbR^3_+$, the homogeneous quadratic map $f^h = (f_{\lambda^{(1)}}^h,f_{\lambda^{(2)}}^h, \ldots, f_{\lambda^{(m)}}^h)$, and the cone $K = -\bbR^m_+$, we have that 
\[X(-\bbR^m_+,f^h) = \left\{[x] \in \bbP^n \; \middle \vert \; x \in \bigcap_{i = 1}^{m}S_{\lambda^{(i)}}^h, x \not = 0\right\}.\]
Note that the condition $f^h([x]) \in K$ is well-defined because $f^h$ is homogeneous quadratic and $K$ is a cone and therefore closed under multiplication by positive scalars. Finally, let $\Omega(K) = K^\circ \cap \bbS^{m-1}$, $C\Omega(K) = K^\circ \cap B^m$, and
\[\Omega^{j}(K) = \{\lambda \in \Omega(K) \;\vert \; Q_\lambda  \text{ has at least } j \text{ positive eigenvalues}\}.\]
We will primarily work with the cases $K = \{0\}$ and $K = -\bbR^3_+$. When $K = \{0\}$ and $K^\circ = \bbR^3$, the set $X(K,f^h)$ is the projective variety $\cV_{\bbR}(f_1^h,f_2^h,f_3^h) \subseteq \bbP^n$ and the sets $\Omega^j(K)$ contain the  (normalized) coefficients of linear combinations of the matrices $Q_i$ with at least $j$ positive eigenvalues. In the case $K = -\bbR^3_+$ and $K^\circ = \bbR^3_+$, $X(K,f^h)$ is the solution set of the system of inequalities $f_i^h([x]) \leq 0$ for $i = 1,2,3$ and the sets $\Omega^j(K)$ contain the (normalized) aggregations corresponding to matrices with at least $j$ positive eigenvalues. 
We omit $K$ from the notation and write e.g. $\Omega(K) = \Omega$ and $X(K,f^h) = X$ when the cone $K$ and homogeneous quadratic map $f^h$ are clear from the context.

We frequently work with homogeneous quadratic maps defined by aggregations $\lambda^{(1)},\lambda^{(2)},\ldots, \lambda^{(m)} \in \Lambda$. In this setting, if $A \in \bbR^{m\times 3}$ is the matrix whose rows are $(\lambda^{(j)})^\top$, then the homogeneous quadratic map $f_A^h = (f_{\lambda^{(1)}}^h,\ldots, f_{\lambda^{(m)}}^h): \bbR^{n+1} \to \bbR^{m}$ factors though $\bbR^3$ as  

\[
\begin{tikzcd}
\mathbb{R}^{n+1} \arrow[r] & \mathbb{R}^3 \arrow[r, "A"]                                  & \mathbb{R}^{m}                   \\
x \arrow[r, maps to]       & f^h(x) = \begin{bmatrix} x^\top Q_1 x\\ x^\top Q_2 x\\ x^\top Q_3 x\end{bmatrix} \arrow[r, maps to] & Af^h(x) = \begin{bmatrix} f_{\lambda^{(1)}}^h(x)\\ \vdots \\ f_{\lambda^{(m)}}^h(x)\end{bmatrix}
\end{tikzcd}
\]
where $f^h = (f_1^h,f_2^h,f_3^h):\bbR^{n+1} \to \bbR^3$ is the homogeneous quadratic map defined by the $Q_i$. 

Now, for a fixed polyhedral cone $K \subseteq \bbR^{k+1}$, the set $X(K,f_A^h) = \{[x] \in \bbP^n \; \vert \; f_A^h(x) \in K\}$ is the same as the set $X(A^{-1}(K),f^h) = \{[x] \in \bbP^n \; \vert \; f^h(x) \in A^{-1}(K)\}$, where $A^{-1}(K) = \{v \in \bbR^3 \; \vert \; Av \in K\}$ is the preimage of $K$ under the map defined by $A$.  Computing polar duals gives that 

\[\begin{aligned}
(A^\top K^\circ)^\circ &= \{v \in \bbR^3\;\vert \; \langle A^\top y, v \rangle \leq 0 \, \text{ for all } y \in K^\circ\}\\
&= \{v \in \bbR^3 \;\vert \; \langle y, Av \rangle \leq 0 \, \text{ for all } y \in K^\circ\}\\
&= \{v \in \bbR^3 \;\vert \; Av \in (K^\circ)^\circ\}\\
&= A^{-1}(K).
\end{aligned}
\]
So, $A^{-1}(K)$ has dual cone $A^\top K^\circ \subseteq \bbR^3$. To compute homology groups of $X(K,f_A^h)$ using Theorem \ref{thm:Spectral_SequenceA}, it therefore suffices to consider dual cones in $\bbR^3$ of the form $A^\top K^\circ$. In particular, we have 

\[X(-\bbR^m_+,f_A^h) = X(\cone(\lambda^{(1)},\lambda^{(2)}, \ldots, \lambda^{(m)})^\circ,f^h).\]

Given a hyperplane $H \subseteq \bbR^{n+1}$, we denote by $Q_\lambda \vert_H$ a matrix representative of the restriction of the quadratic form $f^h_\lambda$ to $H$. In particular, if $B \in \bbR^{(n+1) \times n}$ is a matrix whose columns form an orthonormal basis of $H$ then a matrix representative is given by $B^\top Q_\lambda B$. The signature and singularity of the matrices $Q_\lambda\vert_H$ is independent of the choice of basis of $H$. We denote 
\[\Omega^{j}_H(K) = \left\{\lambda \in \Omega(K) \; \middle \vert \; Q_\lambda \vert_H \text{ has at least } j 
 \text{ positive eigenvalues}\right\}\]
for $j = 0, 1, \ldots, n$. Throughout, (co)homology groups are taken with $\bbZ_2$ coefficients. It is a central result of \cite{agrachev_systems_2012} that homology groups of $X(K,f^h)$ can be computed from the relative cohomology groups of the pairs $(C\Omega(K),\Omega^j(K))$. 

\subsection{Background}

In this section we recall relevant properties of real algebraic plane curves, computations with spectral sequences, and permissible aggregations. 

\subsubsection{Real Algebraic Plane Curves}

Let $h \in \bbR[x,y,z]$ be a homogeneous polynomial of degree $d$. The variety $C = \cV_\bbR(h) \subseteq \bbP^2$ is a \emph{real algebraic plane curve}. A nonsingular real algebraic plane curve $C$ can have multiple connected components. If a connected component $C_1$ of $C$ is contractible in $\bbP^2$ and $\bbP^2 \setminus C_1$ has two connected components, exactly one of which is contractible, then $C_1$ is called an \emph{oval}. The contractible connected component of $\bbP^2 \setminus C_1$ is called the \emph{interior} of $C_1$ and the non-contractible component of $\bbP^2 \setminus C_1$ is the \emph{exterior}. If $d$ is even, then every component of $C$ is an oval. If $d$ is odd, then there is an additional non-oval component. If $C_1, C_2, \ldots, C_k$ are ovals such that $C_i$ is contained in the interior of $C_{j}$ for $1\leq j <i \leq k$ and are the maximal set of ovals satisfying this property, then $C_k$ is said to be an oval of \emph{depth $k$} and the set $C_1,C_2,\ldots,C_k$ is a \emph{nest} of ovals. The paper \cite{degtyarev_number_2012} gives a thorough account of possible arrangements of ovals.

We are particularly interested in the case that there is a nest of ovals in $C$ of maximal depth. This maximal nest must have depth $\lfloor \frac{d}{2} \rfloor$. In this case, the curve $C$ is the hypersurface of a \emph{hyperbolic} polynomial $h \in \bbR[x,y,z]$. Recall that a homogeneous polynomial $h \in \bbR[x,y,z]$ is called hyperbolic with respect to $e \in \bbR^3$ if the univariate polynomial $h(te + a) \in \bbR[t]$ is real rooted for any $a \in \bbR^3$. The connected component of $e$ in $\bbR^3 \setminus \cV_\bbR(h)$ is an open convex cone called a \emph{hyperbolicity cone} of $h$ and $h$ is hyperbolic with respect to any point $e'$ in this cone \cite{garding_inequality_1959}. Note that if $\cP$ is a hyperbolicity cone of $h$, then its image in $\bbP^2$ is the interior of the oval of $C = \cV_\bbR(h)$ of maximal depth. Every hyperbolic plane curve possesses a definite determinantal representation \cite{HeltonVinnikovLMIsets}, but a determinantal representation of a hyperbolic   plane curve is not necessarily definite. See \cite{dym_computing_2012} for an account of determinantal representations of hyperbolic plane curves.  

\subsubsection{Spectral Sequences}

Here we provide a brief introduction to spectral sequences, focusing on the application to computing homology groups rather than an account of the underlying theory. A thorough introduction to spectral sequences is given in \cite{mccleary_users_2001}.

A \emph{(first quadrant cohomology) spectral sequence} $(E_r,d_r)$ is a collection of pages $E_r$ and differentials $d_r:E_r \to E_r$ for $r \geq 0$. Each page $E_r$ consists of a collection of $\bbZ_2$-vector spaces $E_r^{i,j}$ indexed by $\bbZ^2$, where $E_r^{i,j} = 0$ if $i < 0$ or $j < 0$. Each differential is a collection of linear maps $d_r^{i,j}: E^{i,j}_r \to E^{i+r,j-r+1}_r$ such that $d_r \circ d_r = 0$. The $r+1$-st page of the spectral sequence has terms given by the homology of $d_r$. Precisely,

\[E_{r+1}^{i,j} = \frac{\ker(d_r^{i,j})}{\mathrm{im}(d_r^{i-r,j+r-1})}.\]

Since $E_r^{i,j} = 0$ when $i<0$ or $j < 0$, it follows that for sufficiently large $r$, $\ker(d_r^{i,j}) = E^{i,j}_r$ and $\im(d_r^{i-r,j+r-1}) = 0$ so that $E_{r+1}^{i,j} \simeq E_{r}^{i,j}$ and $E_{k}^{i,j} \simeq E_{r}^{i,j}$ for all sufficiently large $k$. In this case, we say that the $i,j$ term of the spectral sequence has \emph{converged} and write $E_{r}^{i,j} = E_{k}^{i,j} = E_{\infty}^{i,j}$. We additionally say that $(E_r,d_r)$ \emph{converges} to some graded $H_*$ if $H_* \simeq \bigoplus_{i+j = *}E_\infty^{i,j}$. 

\begin{example}\label{ex:SpectralSequencePageTurn}
We conclude this section with an example computation of the $E_3$ page of a spectral sequence given the $E_2$ page. This computation forms the basis of our proof of Theorem \ref{thm:empty_var_iff_hyperbolic_curve} below. 

Suppose that $(E_r,d_r)$ is a first quadrant cohomology spectral sequence which converges to $H_{n-*}$. Suppose further that $E_2^{i,j} = \begin{cases} \bbZ_2, & (i,j) = (0,n)\text{ or } (2,n-1)\\0, & \text{otherwise} \end{cases}$ and $d_2^{0,n}:E_2^{0,n} \to E_2^{2,n-1}$ is an isomorphism. The $E_3$ page then has $E_3^{0,n} = \ker(d_2^{0,n})/\im(d_2^{-2,n+1}) = 0/0 \simeq 0$ and $E_3^{2,n-1} = \ker(d_2^{2,n-1})/\im(d_2^{0,n}) = \bbZ_2/\bbZ_2 \simeq 0$. So, the $E_3$ page has $E_3^{i,j} = 0$ for all $i,j$ and therefore $E_\infty^{i,j} = 0$ for all $i,j$. This is represented graphically by 

\begin{sseqdata}[title = $E_\page$, no y ticks, name = Spectral Sequence example, cohomological Serre grading, classes = {draw = none }]
\begin{scope}[background]
    \node at (-1,0) {n-1};
    \node at (-1,1) {n};
\end{scope}
\class["\mathbb{Z}_2"](0,1)
\class["0"](0,0)
\class["0"](1,1)
\class["0"](1,0)
\class["0"](2,1)
\class["\bbZ_2"](2,0)
\d["d_2^{0,n}" description]2(0,1)
\replacetarget["0"]
\replacesource["0"]
\end{sseqdata}

\[
\printpage[ name = Spectral Sequence example, page = 2 ] \qquad \qquad 
\printpage[ name = Spectral Sequence example, page = 3]
\]
Since $E_3^{i,j} = 0$ for all $i,j$, it follows that $E_\infty^{i,j} = 0$ for all $i,j$ and therefore we have that $H_0 \simeq \bigoplus_{i+j = n}E_\infty^{i,j} \simeq 0$. 
\end{example}

\subsubsection{Permissible Aggregations}

We are interested in aggregations $\lambda \in \bbR^3_+$ where the associated set $S_\lambda$ has convex components. This is reflected in the signature of $Q_\lambda$ as follows. 

\begin{proposition}[{\cite[Proposition 4]{yildiran_convex_2009}}] \label{prop:Semiconvex_cone}
 Let $\lambda \in \bbR^3$ and suppose that $\mathrm{int}(S_\lambda^h) \not = \emptyset$. The following are equivalent:

 \begin{enumerate}
    \item There is a linear hyperplane which does not intersect $\mathrm{int}(S_\lambda^h)$.
    \item $Q_\lambda$ has exactly one negative eigenvalue
    \item $\mathrm{int}(S_\lambda^h)$ is a union of two disjoint open convex cones which are symmetric reflections of each other with respect to the origin. 
 \end{enumerate}
\end{proposition}

When the equivalent conditions of Proposition \ref{prop:Semiconvex_cone} are satisfied, $\mathrm{int}(S_\lambda^h)$ is called a \emph{semi-convex cone}. The existence of aggregations with exactly one negative eigenvalue  is reflected in the geometry of the set $S^h$ as follows.

\begin{proposition}\label{prop:Aggs_Exist}
If there exists $\lambda \in \bbR^3_{+}$ such that $Q_\lambda$ has $n$ positive and one negative eigenvalue, then there is a hyperplane $H$ such that $S^h \cap H = 0$. 

Conversely, suppose that $S^h \cap H = 0 $ for some hyperplane $H$ such that the $Q_i\vert_H$ satisfy PDLC. Then, there exists $\lambda \in \bbR^3_+$ such that $Q_\lambda$ has at most one negative eigenvalue.
\end{proposition}

\begin{proof}
For the first statement, take $H$ to be the hyperplane spanned by the $n$ vectors corresponding to positive eigenvalues of $Q_\lambda$. 

For the second statement, if $S^h \cap H = 0$, then $H \cap \mathrm{int}(S^h) = \{(x,x_{n+1}) \in \bbR^{n+1} \cap H \; \vert \; f^h(x,x_{n+1}) < 0\} = \emptyset$.  By the S-lemma variant \cite[Lemma 2.3]{dey_obtaining_2022}, this implies the existence of $\lambda \in \bbR^3_+$ with $Q_\lambda \vert_H \succeq 0$. So, $Q_\lambda$ has at most one negative eigenvalue by the Cauchy Interlacing Theorem. \end{proof}

We are further interested in aggregations $\lambda \in \Lambda$ where $S_\lambda$ satisfies $\conv(S) \subseteq S_\lambda$. We call such $\lambda$ \emph{good aggregations}. An aggregation $\lambda \in \Lambda$ which is not a good aggregation is a \emph{bad aggregation}.

\begin{lemma}[{\cite[Lemma 5.2]{blekherman_aggregations_2022}}]\label{lem:good_aggs}
If $\lambda \in \Lambda$, then $\mathrm{int}(S_\lambda)$ is either a convex set or a union of two disjoint convex sets. Moreover, $\lambda$ is not a good aggregation if and only if $\mathrm{int}(S_\lambda)$ is a union of two disjoint convex sets and $\mathrm{int}(S)$ has nonempty intersection with both components.
\end{lemma}

Note that Lemma \ref{lem:good_aggs} implies that if the set $S$ is connected then every $\lambda \in \Lambda$ is a good aggregation.

\section{Main Results}\label{sec:Main_Results}

In this section we describe our main results in detail and provide examples. 

\subsection{Systems of Quadratics with No Solutions}\label{sec:no_solns}

Our first goal is to apply the results of \cite{agrachev_systems_2012} to certify nonexistence of solutions to systems of quadratics. The primary tool is a spectral sequence relating the homology groups of $X = X(K,f^h)$ to the relative cohomology groups of pairs $(C\Omega,\Omega^j)$. 

\begin{theorem}[{\cite[Theorem A]{agrachev_systems_2012}}]\label{thm:Spectral_SequenceA}
There is a first quadrant cohomology spectral sequence $(E_r,d_r)$ converging to $H_{n - *}(X)$ such that $E_2^{ij} = H^i(C\Omega,\Omega^{j+1})$.
\end{theorem}

Note that the set $C\Omega$ is the intersection of a polyhedral cone and the unit ball and is therefore contractible. So, from the long exact sequence for relative cohomology,

\[\dots \to H^i(C\Omega,\Omega^{j+1}) \to H^i(C\Omega) \to H^i(\Omega^{j+1}) \to H^{i+1}(C\Omega,\Omega^{j+1}) \to \dots,\]
we see that 

\begin{equation}\label{eq:Iso_groups}
E_2^{ij} \cong \begin{cases} H^{i-1}(\Omega^{j+1}), & i \geq 2, \; \Omega^{j+1} \not = \emptyset\\
                        H^0(\Omega^{j+1})/\bbZ_2, & i = 1, \; \Omega^{j+1} \not = \emptyset\\
                        \bbZ_2, & i = 0, \; \Omega^{j+1} = \emptyset\\
                        0, & \text{otherwise}
           \end{cases}.
\end{equation}
The differential $d_2$ has an explicit formula provided by \cite[Theorem B]{agrachev_systems_2012}. Example \ref{ex:E2_example} below demonstrates how the spectral sequence of Theorem \ref{thm:Spectral_SequenceA} and the isomorphisms \eqref{eq:Iso_groups} combine to relate the homology of a variety to the cohomology groups of linear combinations of the defining quadratic forms with specified signature. 

\begin{example}\label{ex:E2_example}
This example is modified from \cite[Example 5.2]{dym_computing_2012}. Consider the matrices

\[Q_1 = \begin{bmatrix} 25 & 0 & -32 & 0\\ 0 & 25 & 0 & 24\\ -32 & 0 & 6 & 0\\ 0 & 24 & 0 & 6 \end{bmatrix}, \quad Q_2 = \begin{bmatrix} 0 & 0 & 12 & 0\\ 0 & 0 & 0 & 16\\ 12 & 0 & 4 & 0\\ 0 & 16 & 0 & 4 \end{bmatrix}, \; \text{and} \; Q_3 = \begin{bmatrix} 0 & 0 & 0 & -60\\ 0 & 0 & 10 & 0\\ 0 & 10 & 0 & 0\\ -60 & 0 & 0 & 0 \end{bmatrix}.\]
It is shown in \cite[Example 5.2]{dym_computing_2012} that $g(\lambda_1,\lambda_2,\lambda_3) = \det(\lambda_1Q_1 + \lambda_2Q_2 + \lambda_3Q_3)$ is smooth and hyperbolic with respect to $(1,0,0)$ but PDLC is not satisfied. For $K = \{0\}$, we compute that $\Omega^1(K) = \bbS^2$, $\Omega^2(K)$ is homotopy equivalent to the union of two points, and $\Omega^3(K)$ is homotopy equivalent to $\bbS^1$. So, the nonzero entries on the $E_2$ page and the potentially nonzero entries on the $E_3$ page are given by

\begin{sseqdata}[title = $E_\page$, name = Hyperbolic example, cohomological Serre grading, classes = {draw = none }, xscale = 2 ]
\class["\mathbb{Z}_2"](0,3)
\class["\bbZ_2"](1,1)
\class["\bbZ_2"](2,2)
\class["\bbZ_2"](3,0)

\d["d_2^{0,3}" description]2(0,3)
\replacesource["\qquad \ker(d_2^{0,3})"]
\replacetarget["\dfrac{\bbZ_2}{\im(d_2^{0,3})}"]
\d["d_2^{1,1}" description]2(1,1)
\replacesource["\ker(d_2^{1,1})"]
\replacetarget["\dfrac{\bbZ_2}{\im(d_2^{1,1})}"]
\end{sseqdata}

\[
\printpage[ name = Hyperbolic example, page = 2 ] \quad
\printpage[ name = Hyperbolic example, page = 3 ]
\]
Moreover, \cite[Example 5.2]{dym_computing_2012} shows that if $f_i^h$ is the quadratic form associated to $Q_i$ for $i = 1,2,3$, then the variety $\cV_\bbR(f_1^h,f_2^h,f_3^h) = \emptyset$. By the computation of the $E_3$ page, and the fact that $d_r = 0$ for $r \geq 3$, we see that $d_2^{0,3}$ and $d_2^{1,1}$ are isomorphisms.  
\end{example}

In the case that $K = \{0\}$ and $X(\{0\},f^h) = \cV_\bbR(f_1^h,f_2^h,f_3^h) \subseteq \bbP^{n}$, the emptiness of $X$ is certified by the hyperbolicity of the spectral curve together with signature information.

\begin{restatable}{theorem}{EmptyVarIFF}\label{thm:empty_var_iff_hyperbolic_curve}
      The projective variety $\cV_\bbR(f_1^h,f_2^h,f_3^h) \subseteq \bbP^{n}$ is empty if $\Omega^n(\{0\}) \not = \emptyset$ (i.e. there is $\mu \in \bbR^3$ such that $Q_\mu$ has at least $n$ positive eigenvalues) and $g$ is smooth and hyperbolic. The converse holds when $n = 1$ or $n \geq 4$.
\end{restatable}

The statement that when $n \geq 4$, emptiness of $X(\{0\},f^h)$ implies the hyperbolicity of $g$ is \cite[Corollary 2]{AgrachevHomologyIntersections}. Note that the case $\Omega^{n+1} \not = \emptyset$ implies that $Q_1,Q_2,Q_3$ satisfy PDLC, which is considered in the context of aggregations of quadratic inequalities in \cite{dey_obtaining_2022}. Additionally, Theorem \ref{thm:empty_var_iff_hyperbolic_curve} is related to the result \cite[Theorem 7.8]{plaumann_quartic_2011}, which characterizes the emptiness of $X = \cV_\bbR(f_1^h,f_2^h,f_3^h)$ in the case $n = 3$. In this case, an empty spectral curve also certifies that $X = \emptyset$. 

The characterization of the emptiness of projective varieties leads to a computation for solution sets of inequalities.

\begin{restatable}{proposition}{GeneralEmptiness}\label{prop:General_Emptiness}
If $X(\{0\},f^h) = \cV_\bbR(f^h_1, f^h_2, f^h_3) = \emptyset$, then for a polyhedral cone $K \not = 0$, the set $X(K,f^h) = \emptyset$ if and only if $\Omega^n(K) \not = \emptyset$ and either $H^1(\Omega^n(K)) \neq 0$ or $\Omega^{n+1}(K) \not = \emptyset$. 
\end{restatable}

The result of Proposition \ref{prop:General_Emptiness} is that non-existnce of solutions of systems of $3$ quadratic inequalities can be certified via convex geometry when the polynomial $g$ is smooth and hyperbolic. Indeed, $H^1(\Omega^n(K)) \not = 0$ means that there is a noncontractible loop in the set of combinations of the matrices $Q_i$ with coefficients in the dual cone $K^\circ$ which have at least $n$ positive eigenvalues. This occurs if and only if PDLC is not satisfied and the hyperbolicity cone of $g$ corresponding to matrices with two negative eigenvalues is contained in $K^\circ$, while $\Omega^{n+1}(K) \not = \emptyset$ if and only if PDLC is satisfied and $K^\circ$ intersects the cone of positive definite matrices. On the other hand, in the exceptional cases $n = 2,3$ where the variety is empty but $g$ is not hyperbolic, any system of inequalities which is not equivalent to the vanishing of all $f_i^h$ has a solution.

\subsection{A Sufficient Condition for Obtaining the Convex Hull via Aggregations}\label{sec:suff_agg}

Here, we apply the topological machinery to develop a new sufficient condition for $\cconv(S)$ to be given by good aggregations. In \cite{blekherman_aggregations_2022,dey_obtaining_2022}, sufficient conditions for the convex hull to be given by aggregations are derived using the following idea. 

\begin{restatable}{proposition}{HyperplaneAggs}\label{prop:HyperplaneAggs}
    Suppose that $\mathrm{int}(S) \not = \emptyset$ and that the matrices $Q_i$ satisfy the following property:

\begin{equation}\tag{*}\label{eq:emptiness_cert}
\begin{split}
&\text{ If } \alpha^\top x < \beta \text{ is a valid inequality on } S, \text{ then for the hyperplane }\\ &H = \{(x,x_{n+1})\; \vert \; \alpha^\top x = \beta x_{n+1}\} \subseteq \bbR^{n+1}, \text{ there exists } \lambda\in \bbR^3_+ \text{ such that } Q_\lambda \vert_{H} \succ 0. 
\end{split}
\end{equation}

Then, there is a set $\Lambda_1 \subseteq \bbR^3_+$ of good aggregations such that $\cconv(S) = \bigcap_{\lambda \in \Lambda_1}S_\lambda $.
\end{restatable}

In \cite{blekherman_aggregations_2022,dey_obtaining_2022}, condition \eqref{eq:emptiness_cert} is imposed via hidden convexity results. In particular, \cite{dey_obtaining_2022} requires PDLC and \cite{blekherman_aggregations_2022} requires hidden hyperplane convexity (i.e. the restriction of $f^h$ to any hyperplane has convex image). Both conditions imply property \eqref{eq:emptiness_cert}.   

Proposition \ref{prop:General_Emptiness} implies that the two possibilities for certificates that $S^h \cap H = 0$ are the existence of a PSD aggregation of $Q_i\vert_H$ or a noncontractible loop in the set of aggregations of the matrices $Q_i\vert_H$ with at least $n-1$ positive eigenvalues (i.e., $H^1(\Omega_H^{n-1}(-\bbR^3_+)) \neq 0$). So, we seek a condition under which every certificate takes the form of an aggregation $\lambda$ with $Q_\lambda\vert_H \succeq 0$. When the polynomials $g(\lambda) = \det(Q_\lambda)$ and $g_H(\lambda) = \det(Q_\lambda\vert_H)$ are both hyperbolic, the first step towards such a condition is understanding the relationship between the hyperbolicity cones $\cP$ of $g$ and $\cP_H$ of $g_H$. If the matrices $Q_i$ satisfy PDLC, then $g_H$ interlaces $g$ so that $\cP \subseteq \cP_H$. If PDLC is not satisfied, the following relationship between $\cP$ and $\cP_H$ holds.

\begin{restatable}{theorem}{HyperbolicityContainment}\label{thm:HyperbolicityContainment}
Suppose $\cV_{\bbR}(f_1^h,f_2^h,f_3^h) = \emptyset$ and $g$ is smooth and hyperbolic with a hyperbolicity cone $\cP$ which contains matrices with exactly two negative eigenvalues. Then, if $g_H$ is smooth and hyperbolic, either $\cP_H$ contains positive definite matrices or $\cP_H \subseteq \cP$. 
\end{restatable}

Combining the ideas in Proposition \ref{prop:HyperplaneAggs} and Theorem \ref{thm:HyperbolicityContainment} suggests that in order to describe $\cconv(S)$ using good aggregations, it suffices to separate the hyperbolicity cone $\cP$ of $g$ from the cone $\bbR_+^3$ of aggregations. We say that the set $S$ \emph{has no points at infinity} if $\{x \in \bbR^n \; \vert \; f_i^h(x,0) \leq 0 \text{ for all } i \in [3]\} = \{0\}$.

\begin{restatable}{theorem}{SuffCond}\label{thm:SuffCond}
    Suppose that $\cV_\bbR(f_1^h,f_2^h,f_3^h) = \emptyset$, the polynomial $g$ is smooth and hyperbolic, $\mathrm{int}(S) \not = \emptyset$, and $S$ has no points at infinity. Let $\cP$ be the hyperbolicity cone of $g$ where $\lambda \in \mathrm{int} \cP$ implies that either $Q_\lambda \succ 0$ or $Q_\lambda$ has exactly two negative eigenvalues. If no nonzero aggregation lies in $\cP$, then $\cconv(S) = \bigcap_{\lambda \in \Lambda_1}S_\lambda$ for some set of good aggregations $\Lambda_1$.
\end{restatable}

We demonstrate by an example that $\cconv(S)$ may not be given by aggregations when a nonzero aggregation lies in $\cP$. 

\begin{example}\label{ex:ch_not_aggs}
This example is modified from \cite[Example 5.2]{dym_computing_2012}. Consider the matrices

\[Q_1 = \begin{bmatrix} 25 & 0 & -32 & 0\\ 0 & 25 & 0 & 24\\ -32 & 0 & 6 & 0\\ 0 & 24 & 0 & 6 \end{bmatrix}, \quad Q_2 = \begin{bmatrix} 0 & 0 & 12 & 0\\ 0 & 0 & 0 & 16\\ 12 & 0 & 4 & 0\\ 0 & 16 & 0 & 4 \end{bmatrix}, \; \text{and} \; Q_3 = \begin{bmatrix} 0 & 0 & 0 & -60\\ 0 & 0 & 10 & 0\\ 0 & 10 & 0 & 0\\ -60 & 0 & 0 & 0 \end{bmatrix}.\]
It is shown in \cite[Example 5.2]{dym_computing_2012} that $g(\lambda_1,\lambda_2,\lambda_3) = \det(\lambda_1Q_1 + \lambda_2Q_2 + \lambda_3Q_3)$ is hyperbolic with respect to $(1,0,0)$ and that the real variety $\cV_\bbR(f_1^h,f_2^h,f_3^h) \subseteq \bbP^3$ is empty.  Here the convex hull of $S$ is not given by aggregations, as shown in Figure \ref{fig:CH_not_Aggs}. 

\begin{figure}[t]
\begin{subfigure}[b]{0.47\textwidth}
    \includegraphics[width = \textwidth]{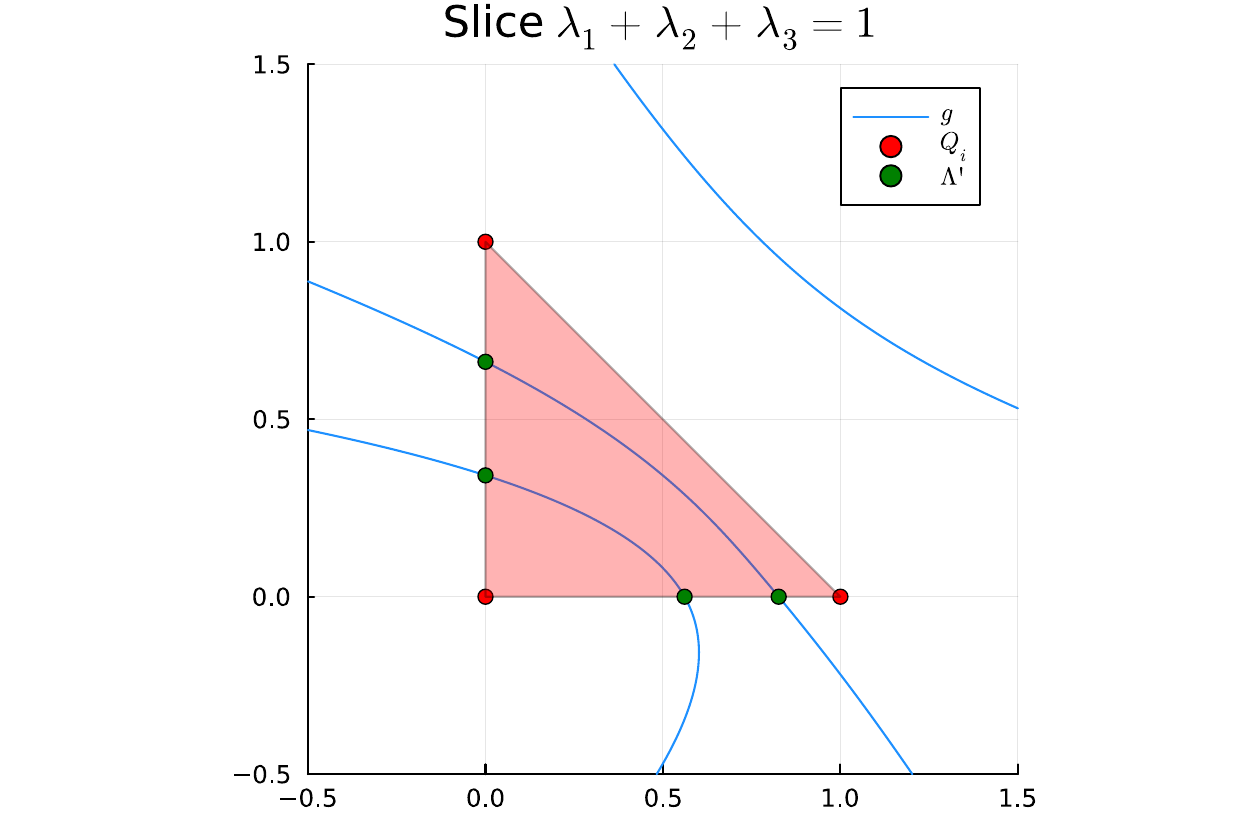}
\end{subfigure}
\begin{subfigure}[b]{0.47\textwidth}
    \includegraphics[width = 0.8\textwidth]{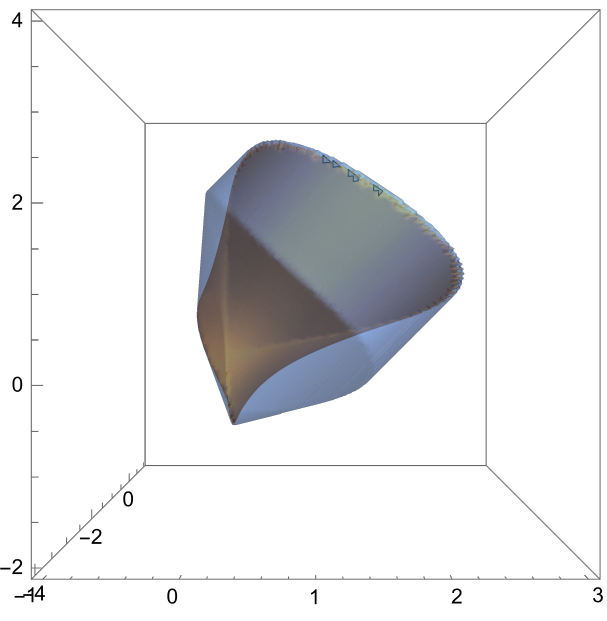}
\end{subfigure}
\caption{The spectral curve and its relation to the defining inequalities in Example \ref{ex:ch_not_aggs} (left). Note that one of the defining inequalities is contained in the hyperbolicity cone of $g$ and that the intersection of $S_\lambda$ for $\lambda \in \Lambda'$ gives a strict superset of $\cconv(S)$ (right).}
\label{fig:CH_not_Aggs}
\end{figure}
\bigskip

On the other hand, if we modify $f_1$ to be 

\begin{equation}\label{eq:modified_ex_yesch}\tilde{f}_1 = \begin{bmatrix} x^\top & 1 \end{bmatrix}\left(0.3Q_1 + 0.4Q_2 + 0.3Q_3\right) \begin{bmatrix} x^\top & 1\end{bmatrix}^\top,
\end{equation}
then the real variety $\cV_\bbR(\tilde{f}_1^h,f_2^h,f_3^h)$ is still empty. However, no nonzero aggregation lies in $\cP$ and the convex hull of the set defined by the inequalities is given by aggregations. This is shown in Figure \ref{fig:CH_given_by_Aggs}.

\begin{figure}[t]
\begin{subfigure}[b]{0.47\textwidth}
    \includegraphics[width = \textwidth]{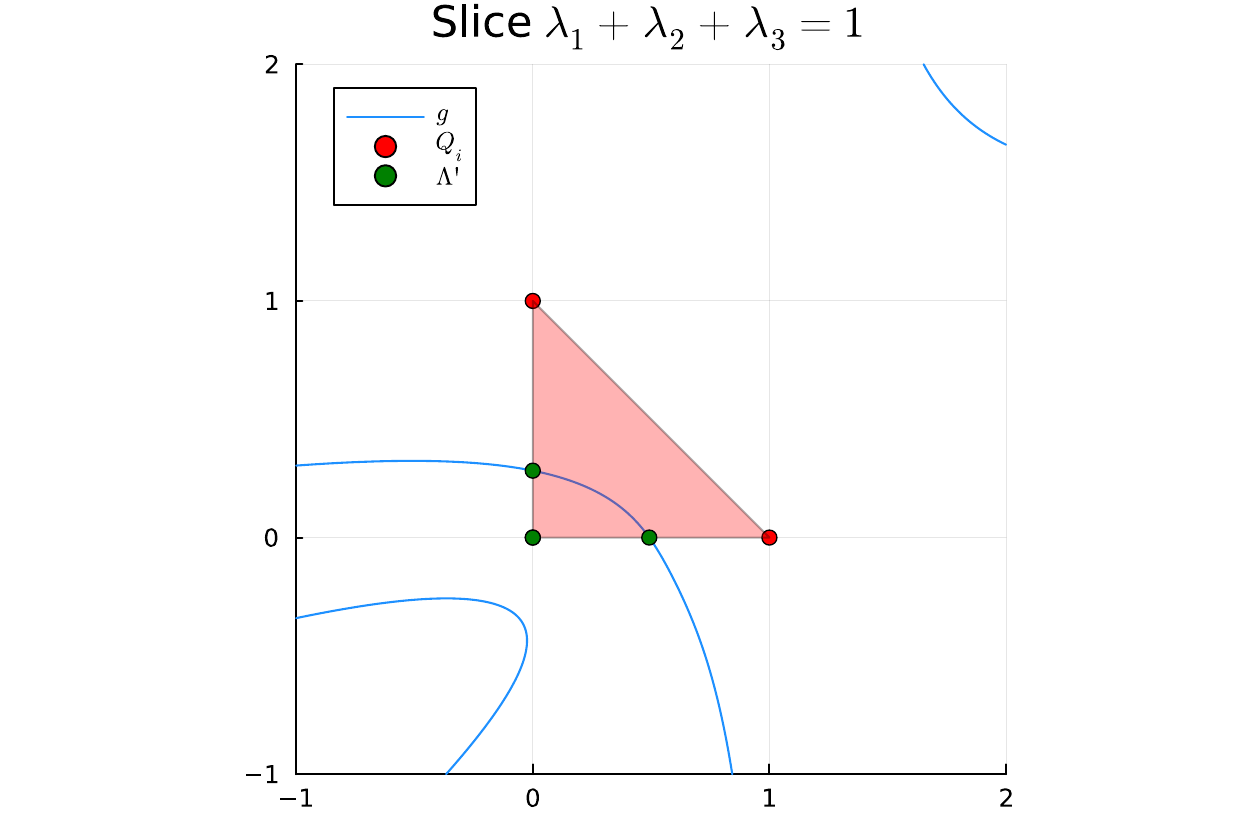}
\end{subfigure}
\begin{subfigure}[b]{0.47\textwidth}
    \includegraphics[width = 0.8\textwidth]{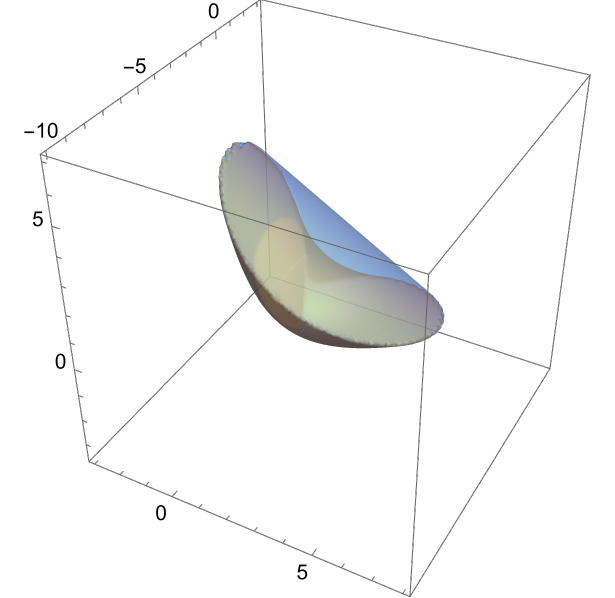}
\end{subfigure}
\caption{The spectral curve and its relation to the modified system of inequalities \eqref{eq:modified_ex_yesch} (left). Note that the conical hull of the defining inequalities has empty intersection with the hyperbolicity cone of $g$ and that the intersection of $S_\lambda$ for $\lambda \in \Lambda'$ gives $\cconv(S)$ (right).}
\label{fig:CH_given_by_Aggs}
\end{figure}

\end{example}

Finally, we consider the case where $n = 2$, $\cV_{\bbR}(f_1^h,f_2^h,f_3^h) = \emptyset$, but $g$ is not hyperbolic. In this setting, we have that for any plane $H \subseteq \bbR^3$, the restrictions $Q_1\vert_H, Q_2\vert_H, Q_3 \vert_H$ satisfy PDLC. It follows that $\cconv(S)$ is given by good aggregations, but it may be necesssary to use infinitely many. 

\begin{theorem}\label{thm:n=2_nonhyperbolic_sufficient}
Suppose that $n = 2$, $\cV_\bbR(f_1^h,f_2^h,f_3^h) = \emptyset$, and $g$ is smooth but not hyperbolic. Then there is a (possibly infinite) set $\Lambda_1 \subseteq \bbR^3_+$ of good aggregations such that $\cconv(S) = \bigcap_{\lambda \in \Lambda_1}S_{\lambda}$.  
\end{theorem}

Theorem \ref{thm:n=2_nonhyperbolic_sufficient} is demonstrated in the following example from \cite{dey_obtaining_2022}.

\begin{example}[{\cite[Proposition 2.8]{dey_obtaining_2022}}]\label{ex:n=2_nonhyperbolic_sufficient}
    Consider the matrices 

    \[Q_1 = \begin{bmatrix} 1 & 0 & 0\\ 0 & 0 & 0\\ 0 & 0 & -1 \end{bmatrix}, \quad Q_2 = \begin{bmatrix} 0 & 0 & 0\\ 0 & 1 & 0\\ 0 & 0 & -1 \end{bmatrix}, \quad \text{and} \quad Q_3 = \begin{bmatrix} -1 & 0 & 1\\ 0 & -1 & 1\\ 1 & 1 & -1 \end{bmatrix},\]
    so that 

    \[S = \{(x,y) \in \bbR^2 \; \vert \; x^2 - 1\leq 0, y^2 - 1 \leq 0, -(x-1)^2 - (y-1)^2 + 1 \leq 0\}.\]
    The determinant is 

    \[g(\lambda) = \det \begin{bmatrix} \lambda_1 - \lambda_3 & 0 & \lambda_3\\ 0 & \lambda_2 - \lambda_3 & \lambda_3\\ \lambda_3 & \lambda_3 & -(\lambda_1 + \lambda_2 + \lambda_3)\end{bmatrix} = -\lambda_1^2\lambda_2 - \lambda_1\lambda_2^2 + \lambda_1^2\lambda_3 + \lambda_1\lambda_2\lambda_3 + \lambda_2^2\lambda_3 - \lambda_1\lambda_3^2 - \lambda_2\lambda_3^2 + \lambda_3^3\]
    The affine cone over the spectral curve is displayed below in Figure \ref{fig:cubic_nonhyperbolic}. It is smooth, but $g$ is not hyperbolic. As shown in \cite[Proposition 2.8]{dey_obtaining_2022}, we have that $\cconv(S)$ is given by the intersection of infinitely many good aggregations and is a proper subset of the intersection of any finite number of good aggregations. 

    \begin{figure}[t]
        \centering
        \includegraphics[width = 0.5\textwidth]{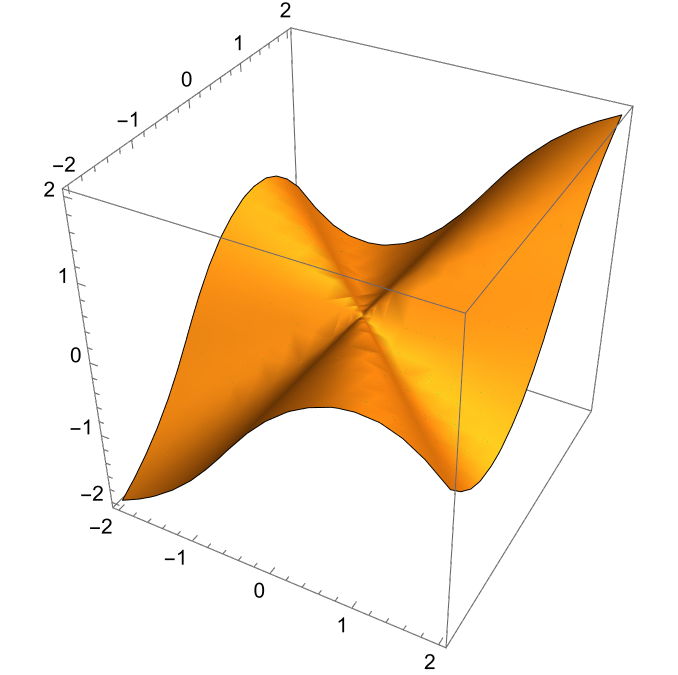}
        \caption{The affine cone over the spectral curve in Example \ref{ex:n=2_nonhyperbolic_sufficient}. The spectral curve is smooth but not hyperbolic. In this case, $\cconv(S)$ is obtained as the intersection of good aggregations, but any finite subset of good aggregations will give a strict superset of $\cconv(S)$.} 
        \label{fig:cubic_nonhyperbolic}
    \end{figure}
\end{example}

\subsection{Bounds on the Number of Needed Aggregations}\label{sec:AggsFinite}

Using the tools developed in Section \ref{sec:no_solns}, we show that if the real projective variety $\cV_{\bbR}(f_1^h,f_2^h,f_3^h)$ is empty and the spectral curve is hyperbolic, then a finite number of permissible aggregations suffices to recover the set defined by all permissible aggregations. An immediate reduction is that it suffices to consider only unit length generators of the extreme rays of the set $\Lambda$. To see this, note that $S_\lambda = S_{t\lambda}$ for $t > 0$ and if $\omega \in \Lambda$ and $\omega = t\lambda^{(1)} + (1-t)\lambda^{(2)}$ for some $t \in (0,1)$ and $\lambda^{(1)},\lambda^{(2)} \in \Lambda$, then $S_\omega \subseteq S_{\lambda^{(1)}} \cap S_{\lambda^{(2)}}$ since

\[f_\omega(x) = f_{t\lambda^{(1)} +(1-t)\lambda^{(2)}}(x) = tf_{\lambda^{(1)}}(x) + (1-t)f_{\lambda^{(2)}}(x) \leq 0.\]

We further reduce the number of needed aggregations by applying the following property.

\begin{restatable}{lemma}{UnnecessaryAggstwo}\label{lem:UnnecessaryAggstwo}
    Let $\lambda^{(1)},\lambda^{(2)},\ldots, \lambda^{(k+1)} \in \Lambda$. Set $f_k^h = (f_{\lambda^{(1)}}^h, \ldots, f_{\lambda^{(k)}}^h)$ and $f_{k+1}^h = (f_{\lambda^{(1)}},\ldots, f_{\lambda^{(k+1)}}^h)$. Suppose that $\cV_{\bbR}(f_{\lambda^{(k+1)}}^h) \cap X(-\bbR^{k+1}_+, f_{k+1}^{h}) = \emptyset$ and that $X(-\bbR^{k+1}_+,f_{k+1}^h)$ has the same number of connected components as $X(-\bbR^k_+,f_k^h)$. Then,

    \[X(-\bbR^{k+1}_+,f_{k+1}^h) = X(-\bbR^{k}_+,f_{k}^h).\]
    
    \end{restatable}

Lemma \ref{lem:UnnecessaryAggstwo} allows us to conclude that an aggregation is not necessary by applying Proposition \ref{prop:General_Emptiness} with a cone of the form $K = \bbR^k \times \{0\}$. By the hypothesis that the determinantal polynomial $g$ is hyperbolic, the computation of $H^1(\Omega^n(K))$ reduces to arguments about separating a polyhedral cone from a hyperbolicity cone of $g$. Let $\mathrm{ex}(\Lambda)$ be the set of generators of extreme rays of $\Lambda$ with unit length and set

\[\Lambda' = \{\lambda \in \mathrm{ex}(\Lambda) \; \vert \; |\supp(\lambda) | \leq 2 \}.\]

Note that $|\Lambda'| \leq 6$ since there can be at most two extreme rays of $\Lambda$ contained in each facet of $\bbR^3_+$. Under a condition that ensures connectedness, $\Lambda'$ suffices to describe all permissible aggregations. Specifically, we enforce that $\Omega^n(\cone(\Lambda')^\circ)$ is contractible, i.e, the set of matrices $Q_\lambda$ for $\lambda \in \cone(\Lambda') \cap \bbS^2$ which have at least $n$ positive eigenvalues is connected and has no higher cohomology. Using the spectral sequence of Theorem \ref{thm:Spectral_SequenceA}, we see that the condition that $\Omega^n(\cone(\Lambda')^\circ)$ is contractible is met when $\bigcap_{\lambda \in \Lambda'}S_\lambda$ is connected. In particular, $\Omega^n(\cone(\Lambda')^\circ)$ is contractible when $\bigcap_{\lambda \in \Lambda'}S_\lambda$ is convex.  

\begin{restatable}{theorem}{FiniteAggs}\label{thm:FiniteAggs}
    Suppose that $\cV_{\bbR}(f_1^h,f_2^h,f_3^h) \subseteq \bbP^n$ is empty, the spectral curve is smooth and  hyperbolic, and that $\Omega^n(\cone(\Lambda')^\circ)$ is contractible. Then, 

\[\bigcap_{\lambda \in \Lambda'}S_\lambda = \bigcap_{\lambda \in \Lambda}S_\lambda.\]
\end{restatable}

We show in Example \ref{ex:positive_agg_needed} that if the contractibility condition on $\Omega^n(\cone(\Lambda')^\circ)$ is not satisfied, then it may be necessary to use an aggregation with strictly positive entries. In  this case, the aggregation with strictly positive entries has the effect of cutting off a connected component of $\bigcap_{\lambda \in \Lambda'}S_\lambda$. 

\begin{example}\label{ex:positive_agg_needed}
    We construct a system of quadratics such that $\bigcap_{\lambda \in \Lambda'} S_\lambda$ has three connected components, but for some $\mu$ with strictly positive entries, $S_\mu \cap \left(\bigcap_{\lambda \in \Lambda'}S_\lambda\right)$ has only two components. As in Example \ref{ex:E2_example}, consider the matrices 

    \[M_1 = \begin{bmatrix} 25 & 0 & -32 & 0\\ 0 & 25 & 0 & 24\\ -32 & 0 & 6 & 0\\ 0 & 24 & 0 & 6 \end{bmatrix}, \quad M_2 = \begin{bmatrix} 0 & 0 & 12 & 0\\ 0 & 0 & 0 & 16\\ 12 & 0 & 4 & 0\\ 0 & 16 & 0 & 4 \end{bmatrix}, \; \text{and} \; M_3 = \begin{bmatrix} 0 & 0 & 0 & -60\\ 0 & 0 & 10 & 0\\ 0 & 10 & 0 & 0\\ -60 & 0 & 0 & 0 \end{bmatrix}.\]

    The matrices do not satisfy PDLC and $\det(xM_1 + yM_2 + zM_3)$ is hyperbolic with respect to $(1,0,0)$. We construct the defining quadratics $f_i(x,y,z) = \begin{bmatrix}x & y & z & 1 \end{bmatrix} Q_i \begin{bmatrix} x & y & z & 1\end{bmatrix}^\top$ with $Q_i$ defined as follows:

    \[\begin{aligned}
    Q_1 = M_1 + 1.5M_2  =\begin{bmatrix} 25& 0 & -14 & 0\\ 0 & 25 & 0 & 48\\ -14 & 0 & 12 & 0\\ 0 & 48 & 0 & 12\end{bmatrix}\\
    Q_2 = M_1 -2M_2 + 2M_3 = \begin{bmatrix}25 & 0 & -56 & -120\\ 0 & 25 & 20 & -8\\ -56 & 20 & -2 & 0\\ -120 & -8 & 0 & -2 \end{bmatrix}\\
    Q_3 = M_1 -2M_2 - 2M_3 = \begin{bmatrix}25 & 0 & -56 & 120\\ 0 & 25 & -20 & -8\\ -56 & -20 & -2 & 0\\ 120 & -8 & 0 & -2 \end{bmatrix}
    \end{aligned}\]

    We numerically compute that the elements of $\Lambda'$ are appropriately normalized vectors in the directions of the following vectors $\lambda^{(1)},\lambda^{(2)}, \lambda^{(3)}$. Since the semialgebraic set $S_\lambda$ is invariant under positive scalings of $\lambda$, we work directly with the $\lambda^{(i)}$.  We also take an aggregation $\mu$ which lies on the oval of depth $\lfloor\frac{n+1}{2} \rfloor -1$ with strictly positive entries. 
    
    \[\lambda^{(1)} = e_1,\quad \lambda^{(2)} = (1,0.68725,0), \quad \lambda^{(3)} = (1,0,0.68725), \quad \mu = (0.1429,0.4286,0.4286)\]

    The intersection of the cone $K = \cone(\lambda^{(1)},\lambda^{(2)},\lambda^{(3)})$ with $\Omega^{n}(-\bbR^3_+)$ has three connected components and the intersection of the cone $\tilde{K} = \cone(\lambda^{(1)},\lambda^{(2)},\lambda^{(3)},\mu)$ with $\Omega^{n}(-\bbR^3_+)$ has two connected components. The corresponding semialgebraic sets $\bigcap_{i =1}^{3} S_{\lambda^{(i)}}$ and $S_\mu \cap \left(\bigcap_{i = 1}^3S_{\lambda^{(i)}}\right)$ have three and two connected components, respectively. However, as in the proof of Theorem \ref{thm:FiniteAggs}, $\cV_{\bbR}(f_\mu^h) \cap \left(\bigcap_{i =1}^{3} S_{\lambda^{(i)}}\right) = \emptyset$. So intersecting with $S_\mu$ has the effect of cutting off a component of $\left(\bigcap_{i =1}^{3} S_{\lambda^{(i)}}\right)$. In particular, we see that $\bigcap_{\lambda \in \Lambda'}S_\lambda \not = \bigcap_{\lambda \in \Lambda}S_\lambda$. The spectral curve and the relevant semialgebraic sets are shown in Figure \ref{fig:Need_positive_agg}.

    \begin{figure}[t]
    \centering
    \begin{subfigure}[b]{0.3\textwidth}
    \includegraphics[width = \textwidth]{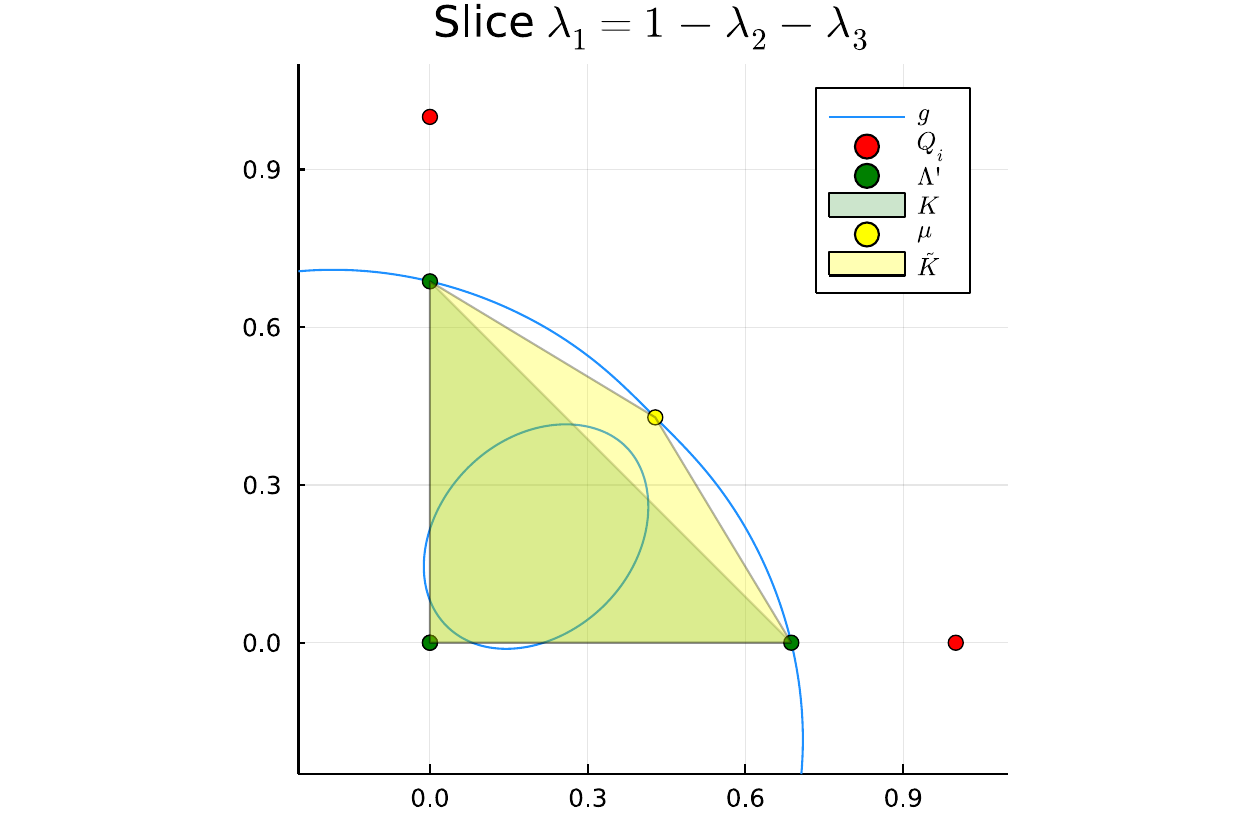}
    \end{subfigure}
    \begin{subfigure}[b]{0.3\textwidth}
    \includegraphics[width = 0.8\textwidth]{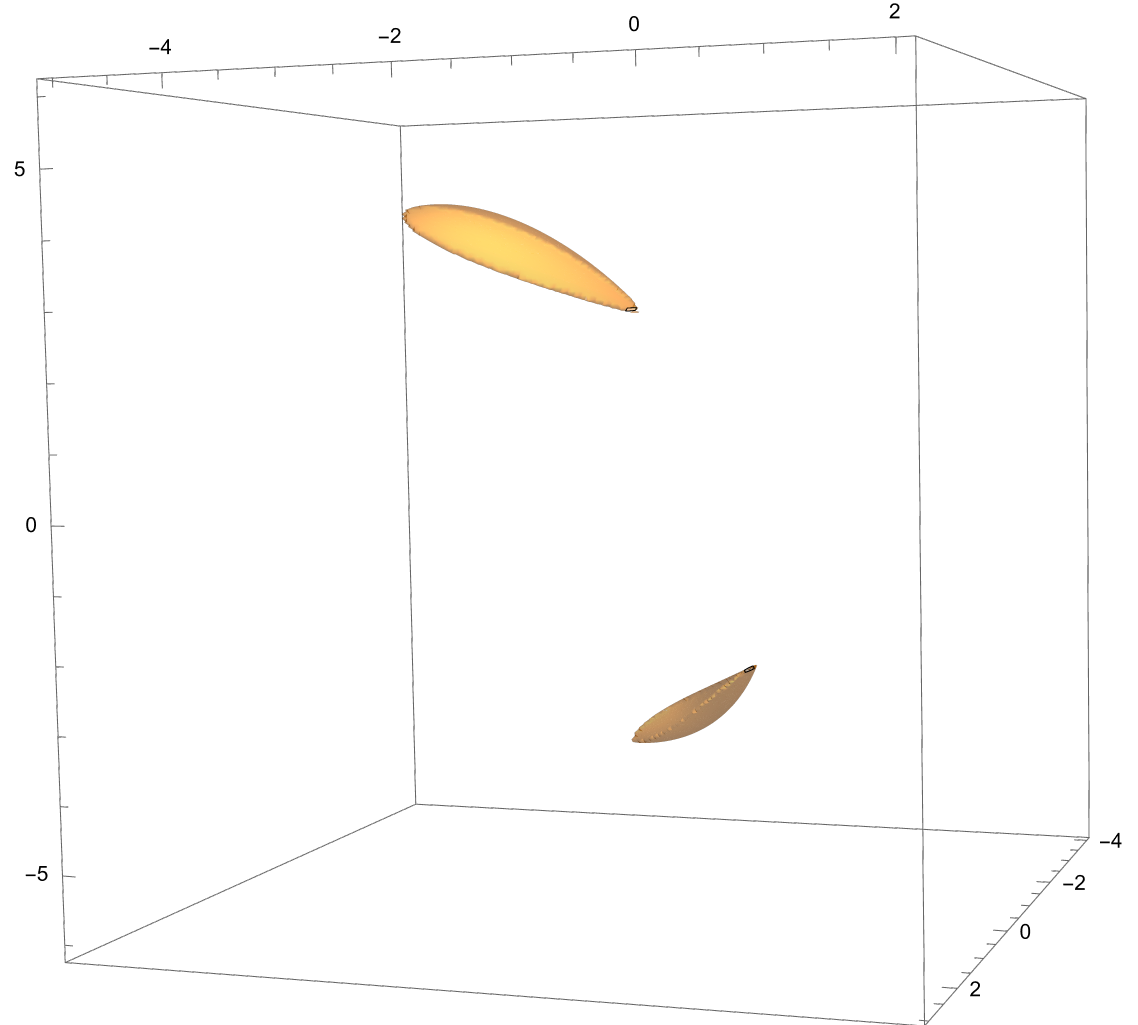}
    \end{subfigure}
    \begin{subfigure}[b]{0.3\textwidth}
    \includegraphics[width = 0.8\textwidth]{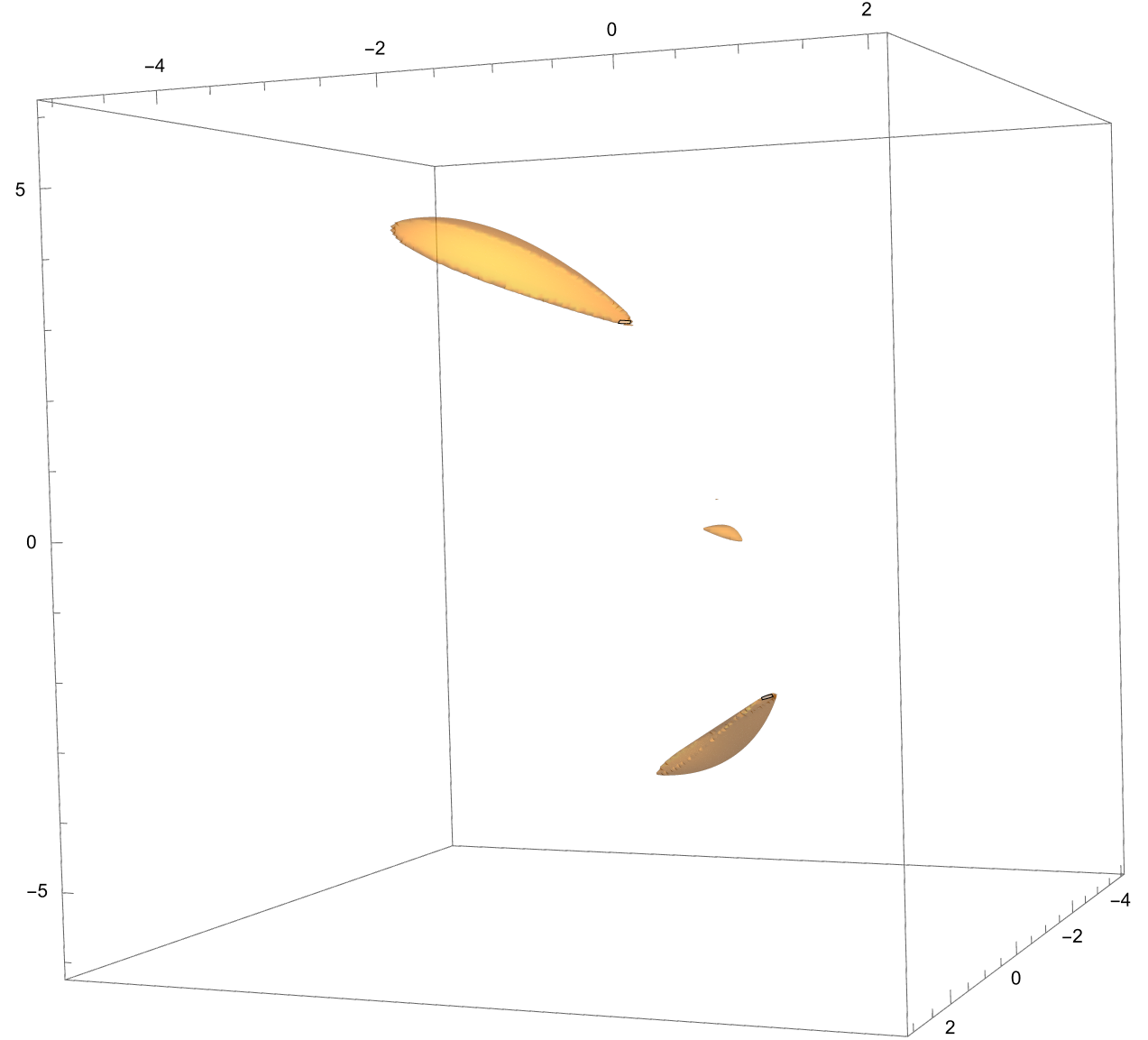}
    \end{subfigure}
    \caption{Plots for Example \ref{ex:positive_agg_needed}. The left figure displays the spectral curve and relevant aggregations and cones. The center figure shows the set $S_\mu \cap \left(\bigcap_{\lambda \in \Lambda'}S_\lambda\right)$ which has two connected components. The right figure shows $\bigcap_{\lambda \in \Lambda'}S_\lambda,$ which has three connected components.}
    \label{fig:Need_positive_agg}
    \end{figure}
\end{example}

The construction in Example \ref{ex:positive_agg_needed} inspires the following bound on the number of needed aggregations in terms of the number of connected components of the set $\Omega^n(\cone(\Lambda')^\circ)$.

\begin{restatable}{theorem}{NoPDLCDisconnectBound}\label{thm:NoPDLCDisconnectBound}
    Suppose that $\cV_{\bbR}(f_1^h,f_2^h,f_3^h) \subseteq \bbP^n$ is empty,
    the spectral curve is smooth and hyperbolic, and  $\Omega^n(\cone(\Lambda')^\circ)$ is not contractible. Then, there is a set $\Lambda_1 \subseteq \Lambda$ with $|\Lambda_1| \leq |\Lambda'| + \dim(H^0(\Omega^n(\cone(\Lambda')^\circ))) - 1$ such that 

    \[\bigcap_{\lambda \in \Lambda_1}S_\lambda = \bigcap_{\lambda \in \Lambda} S_\lambda.\]
\end{restatable}

\subsection{Topology of Good Aggregations}\label{sec:Agg_top}

Our final goal is an investigation of the topology of good and bad aggregations. 

\begin{restatable}{proposition}{ConnectedGoodAggs}\label{prop:ConnectedGoodAggs}
    Suppose that $S = \mathrm{cl}(\mathrm{int}(S))$. Then, every connected component of $\Lambda$ consists of only good aggregations or of only bad aggregations.
\end{restatable}

The assumption that $S$ satisfies the regularity condition $S = \mathrm{cl}(\mathrm{int}(S))$ cannot be removed, as demonstrated by the following example.

\begin{example}\label{ex:NonRegular_aggstwocomponents}
This example is modified from \cite[Example 2.22]{blekherman_aggregations_2022} Consider the three matrices 

\[Q_1 = \begin{bmatrix}-1 & 0 & 0.5\\ 0 & 0 & 0 \\ 0.5 & 0 & 0\end{bmatrix}, \qquad Q_2 = \begin{bmatrix}1 & 0 & 0\\ 0 & 1 & 0 \\ 0 & 0 & -1\end{bmatrix}, \quad \text{ and } \quad Q_3 = \begin{bmatrix} 1 & 0 & 0\\ 0 & 1 & 0 \\ 0 & 0 & -3\end{bmatrix}\]
and the associated quadratic functions $f_i$. Note that the linear combination

\[Q_1 + 4Q_2 -2Q3 = \begin{bmatrix} 1 & 0 & 0.5\\ 0 & 2 & 0 \\ 0.5 & 0 & 2\end{bmatrix}\]
is positive definite. The set 

\[S = \left\{x = (x_1,x_2) \in \bbR^2 \; \vert \; \|x\|_2^2 \leq 1, x_1 \leq 0 \right\} \cup \{e_1\}\]
does not satisfy $S = \mathrm{cl}(\mathrm{int}(S))$. Now, $\Lambda$ is connected, but there exist both good and bad aggregations ($(0.5,0.5,0)$ is a good aggregation while $(0.75,0.25,0)$ is a bad aggregation). On the other hand, $\Omega^n(-\bbR^3_+)$ has two connected components.  
Additionally, note that every element $\lambda \in \Lambda$ satisfies the property $\conv(\mathrm{int}(S)) \subseteq \mathrm{int}(S_\lambda)$. 
\end{example}

Finally, under the assumption of PDLC, there is a unique connected component of $\Omega^n(-\bbR^3_+)$ consisting of good aggregations.

\begin{restatable}{proposition}{GoodAggsConnectedPDLC}\label{prop:GoodAggsConnectedPDLC}
    If $Q_1,Q_2,Q_3$ satisfy PDLC, $\mathrm{int}(S) \not = \emptyset$, and $S$ has no points at infinity, then exactly one connected component of $\Omega^n(-\bbR^3_+)$ consists of good aggregations.
\end{restatable}

Using the topology of good aggregations as a subset of all permissible aggregations, we prove Theorem \ref{thm:PDLCBound}. Specifically, under PDLC and regularity hypotheses on $S$, there is a set of $r\leq 4$ good aggregations $\lambda^{(1)}, \lambda^{(2)}, \ldots, \lambda^{(r)}$ such that $\bigcap_{i = 1}^{r}S_{\lambda^{(i)}}$. The proof is in Section \ref{sec:Connected_Aggs}.

\section{Conclusions}\label{sec:Conclusions}

We have presented an analysis of aggregations of three quadratic inequalities under the assumption that the projective variety of the defining quadratics is empty. In particular, we provide a sufficient condition for the description of $\cconv(S)$ via good aggregations which does not depend on PDLC. We show that the semialgebraic set defined by the intersection of all permissible aggregations can be described by the intersection of finitely many permissible aggregations when the spectral curve is hyperbolic. If the set is connected, then we can explicitly compute these aggregations. 

Our approach differs from previous approaches to this problem in that we heavily utilize tools from algebraic topology. These tools are infrequent in the optimziation literature, but this work demonstrates that these topological tools can have useful applications. Our analysis is limited to the case of three inequalities and an empty variety, as we heavily rely on the structure of hyperbolic plane curves. We leave open the case of more defining inequalities or a nonempty variety. In these cases, the spectral hypersurface may be non-hyperbolic or not a plane curve, preventing a direct translation of our results.

\section{Spectral Sequence Computations}\label{sec:Spectral_comps}

In this section, we present the spectral sequence computations used to derive the results in Section \ref{sec:no_solns}. We first concentrate on the case that $n \geq 4$, $K = \{0\}$, and $X = \cV(f_1^h,f_2^h,f_3^h) \subseteq \bbP^n$ and prove that the emptiness of $X$ implies hyperbolicity of $g(\lambda) = \det(Q_\lambda)$. Note that the case $n = 3$ is proved in \cite[Theorem 7.8]{plaumann_quartic_2011}. 

We begin by relating the signature of a matrix $Q_\lambda$ to the location of $\lambda$ relative to the ovals of the spectral curve $\cV_\bbR(g)$ when it is smooth and nonempty. 

\begin{lemma}\label{lem:signature_ovals}
If $\lambda \in \Omega^n(\{0\})$ and the spectral curve $\cV_\bbR(g)$ is smooth, then $[\lambda] \in \bbP^2$ lies in the interior of an oval of $\cV_\bbR(g)$ of depth at least $\lfloor \frac{n+1}{2}\rfloor - 1$.
\end{lemma}

\begin{proof}
Let $x \in \bbS^2$ be such that $[x] \in \bbP^2$ lies on the exterior of every oval of $\cV_\bbR(g)$. By \cite{Vinnikov1993SelfAdjoint}, if $n+1$ is even $Q_x$ has $\frac{n+1}{2}$ positive and negative eigenvalues and if $n+1$ is odd, one of $Q_x$ and $-Q_x$ has $\lfloor \frac{n+1}{2} \rfloor + 1$ positive and $\lfloor \frac{n+1}{2} \rfloor $ negative eigenvalues. Since $g$ is nonsingular, it follows that if $[\lambda]$ is in the interior of an oval of depth $k$ and the exterior of all ovals of depth $k+1$, then $Q_\lambda$ has at most $\frac{n+1}{2} + k$ positive eigenvalues if $n+1$ is even and $\lfloor \frac{n+1}{2}\rfloor + k + 1$ positive eigenvalues if $n+1$ is odd. So, if $Q_\lambda$ has $n$ positive eigenvalues and $n+1$ is even, $[\lambda]$ lies in the interior of an oval of depth at least $\frac{n+1}{2} -1$ and if $n+1$ is odd, $[\lambda]$ is an element of the interior of an oval of depth at least $\lfloor \frac{n+1}{2} \rfloor -1$
\end{proof}

The spectral sequence of Theorem \ref{thm:Spectral_SequenceA} then gives the following. 

\begin{lemma}[cf {\cite[Corollary 2]{ AgrachevHomologyIntersections}}]\label{lem:Empty_variety_hyperbolic_curve}
Suppose that $n \geq 4$ and the spectral curve is smooth. If $X = \cV_{\bbR}(f_1^h,f_2^h,f_3^h) \subseteq \bbP^n$ is empty, then $g$ is hyperbolic and $\Omega^n(\{0\}) \not =  \emptyset$. 
\end{lemma}

\begin{proof}

For notational simplicity, take $X = X(\{0\},f^h)$ and $\Omega^j = \Omega^j(\{0\})$. Let $(E_r,d_r)$ be the spectral sequence of Theorem \ref{thm:Spectral_SequenceA}.

First, suppose that the matrices $Q_i$ satisfy PDLC so that there is $\lambda \in \bbR^3$ with $Q_\lambda \succ 0$. Then $g$ is hyperbolic with respect to $\lambda$ and $Q_\lambda$ has $n+1 > n$ positive eigenvalues so that $\Omega^n \not = \emptyset$. 

Next we show that if the maximum number of positive eigenvalues attained by $Q_\lambda$ for $\lambda \in \bbR^3$ is at most $n-1$, then $X \not = \emptyset$. 
Using \eqref{eq:Iso_groups}, we see that $E^{ij}_2 = H^{i-1}(\Omega^{j+1}) =0$ for $i \geq 4$ and that $E_2^{2,n-1} = 0$. Note that $\Omega^{n-1}$ is either empty or a proper subset of $\mathbb{S}^2$. This follows since $n \geq 4$ and therefore if $\lambda$ is such that $Q_\lambda$ has at least $n-1$ positive eigenvalues, then $Q_{-\lambda}$ has at most $(n+1) - (n-1) =  2 < n-1$ positive eigenvalues. So $E_2^{3,n-2} = H^2(\Omega^{n-1}) = 0$. This implies that $E_2^{i,n-i+1} = 0$ for all $i \geq 2$ and therefore $E_2^{0,n} \simeq E_\infty^{0,n} \simeq \bbZ_2$. So, $H_0(X) \not = 0$ and therefore $X$ is nonempty. 

Finally, suppose that there exists $\lambda \in \bbR^3$ such that $Q_\lambda$ has exactly $n$ positive eigenvalues and that PDLC is not satisfied. We apply Theorem \ref{thm:Spectral_SequenceA} and the hypothesis that $X = \emptyset$ to compute the groups $H^i(\Omega^j)$. Note that $Q_{-\lambda}$ has at most 1 positive eigenvalue so that $\Omega^{j+1}$ is a proper subset of $\bbS^2$ for $j \geq 1$. Since $n \geq 4$, this implies that $E^{3,j}_2 \simeq H^2(\Omega^{j+1}) = 0$ for $j \geq n-3$. So, the $E_2$ page of the spectral sequence has the following form:

\begin{sseqdata}[title = $E_\page$, no y ticks, name = empty var top four, cohomological Serre grading, classes = {draw = none }, xscale = 2]
\begin{scope}[background]
    \node at (-0.6,0) {n-3};
    \node at (-0.6,1) {n-2};
    \node at (-0.6,2) {n-1};
    \node at (-0.6,3) {n};
\end{scope}
\class["\mathbb{Z}_2"](0,3)
\class["0"](0,0)
\class["0"](0,1)
\class["0"](0,2)

\class["0"](1,3)
\class["H^0(\Omega^n)/\bbZ_2"](1,2)
\class["H^0(\Omega^{n-1})/\bbZ_2"](1,1)
\class["H^0(\Omega^{n-2})/\bbZ_2"](1,0)

\class["0"](2,3)
\class["H^1(\Omega^n)"](2,2)
\class["H^1(\Omega^{n-1})"](2,1)
\class["H^1(\Omega^{n-2})"](2,0)

\foreach \y in {0,...,3}{
    \class["0"](3,\y)
}
\end{sseqdata}

\[
\printpage[name = empty var top four  , page = 2] 
\]

Note that $E_2^{1,n-1} \simeq E_\infty^{1,n-1} \simeq H^0(\Omega^n)/\bbZ_2$ and $E_2^{2,n-2} \simeq E_\infty^{2,n-2} \simeq H^1(\Omega^{n-1})$. Since $X = \emptyset$, Theorem \ref{thm:Spectral_SequenceA} implies that

\[0 = H_{0}(X) \simeq \bigoplus_{i+j = n}E^{i,j}_\infty.\]
In particular, $H^0(\Omega^n)/\bbZ_2 = 0$, $H^1(\Omega^{n-1}) = 0$, and the map $d_2^{0,n}:\bbZ_2 \to H^1(\Omega^n)$ is injective.  By Lemma \ref{lem:signature_ovals}, $\Omega^n$ is a subset of the interior of ovals of depth at least $\lfloor \frac{n+1}{2}\rfloor - 1$. Since $H^1(\Omega^n) \not = 0$, there is anest ovals of depth $\lfloor \frac{n+1}{2} \rfloor$. So, the curve $\cV_{\bbR}(g)$ is hyperbolic. \end{proof}

For the converse statement, note that if PDLC holds, then the variety is empty. The remaining case is that $g$ is hyperbolic with a hyperbolicity cone $\cP$ which consists of matrices with exactly $n-1$ positive and two negative eigenvalues. In this case, we must compute the differential $d_2^{0,n}$. In particular, we want to show that $d_2^{0,n}:\bbZ_2 \to H^1(\Omega^n)$ is injective. Let $\mathcal{Q}$ be the vector space of quadratic forms on $\bbR^{n+1}$. For $q \in \mathcal{Q}$, let $\rho_1(q) \geq \rho_2(q) \geq \ldots \geq \rho_{n+1}(q)$ be the eigenvalues of $q$ and $\cD_j = \{q \in \mathcal{Q} \; \vert \; \rho_j(q) > \rho_{j+1}(q)\}.$

\begin{theorem}[{\cite[Theorem B]{agrachev_systems_2012}}]\label{thm:Spectral_Sequence_d2}
There is a formula for the differential $d_2$ in terms of matrices with repeated eigenvalues. Explicitly, 

\[d_2(x) = (x \smile \bar{f}^* \gamma_{1,j})\vert_{(C\Omega, \Omega^j)} \quad \text{ for } x \in H^*(C\Omega,\Omega^{j+1}),\]
where $\bar{f}:\bbR^3 \to \mathcal{Q}$ is defined by $\bar{f}(\lambda) = f_\lambda^h$, and $\bar{f}^*$ is the induced map on cohomology. The value of the class $\gamma_{1,j} \in H^2(\mathcal{Q},\cD_j)$ on the image of $\sigma:B^2 \to \mathcal{Q}$ with $\sigma(\partial B^2) \subseteq \cD_j$ is equal to the intersection number of $\sigma(B^2)$ and $\mathcal{Q} \setminus \cD_j \mod 2$.
\end{theorem}

In light of Theorem \ref{thm:Spectral_Sequence_d2}, the differential $d_2^{0,n}$ is computed by understanding the set of matrices with repeated negative eigenvalue inside the hyperbolicity cone $\cP$ of $g$. To show that $d_2^{0,n}$ is injective, it suffices to show that there is exactly one $\lambda\in \mathrm{int} \cP \cap \bbS^2$ with repeated negative eigenvalue. In this case, it will follow that $\gamma_{1,n}$ and therefore $d_2^{0,n}(1)$ are nonzero. We proceed in several steps: Lemma $\ref{lem:Repeated_neg_on_pencil}$ and Proposition \ref{prop:Unique_neg_rep} give the existence and uniqueness of  $\lambda\in \mathrm{int} \cP \cap \bbS^2$ with repeated negative eigenvalue, which is leveraged for the computation of $d_2^{0,n}(1)$ in the proof of Theorem \ref{thm:empty_var_iff_hyperbolic_curve}.

We first show that if the matrices corresponding to two points in $\cP$ have a repeated negative eigenvalue, then the matrices corresponding to every point along the affine pencil connecting them has at least 2 negative eigenvalues.

\begin{lemma}\label{lem:Repeated_neg_on_pencil}
Suppose that $\Omega^n \not = \emptyset$ and $g$ is smooth and hyperbolic with hyperbolicity cone $\cP$ containing matrices with exactly two negative eigenvalues. If $\lambda^{(1)}, \lambda^{(2)} \in \cP$ are such that $Q_{\lambda^{(1)}}$ and $Q_{\lambda^{(2)}}$ have repeated negative eigenvalue $-1$, then for each $t \in \bbR$, $tQ_{\lambda^{(1)}} + (1-t)Q_{\lambda^{(2)}}$ has at least 2 negative eigenvalues. 
\end{lemma}

\begin{proof}
First, note that the statement holds for all $t \in [0,1]$ by convexity of $\cP$. Let $V$ be the two dimensional subspace corresponding to the eigenvalue $-1$ of $Q_{\lambda^{(1)}}$ and $U$ be the two dimensional subspace corresponding to the eigenvalue $-1$ of $Q_{\lambda^{(2)}}$. Then, for $t \geq 1$, and for all $v \in V$ with $v^\top v = 1$, we have that 

\[v^\top(tQ_{\lambda^{(1)}} + (1-t)Q_{\lambda^{(2)}})v = -t + (1-t)v^\top Q_{\lambda^{(2)}} v \leq -1,\]
where the inequality holds since $-1$ is the lowest eigenvalue of $Q_{\lambda^{(2)}}$ and therefore $v^\top Q_{\lambda^{(2)}}v \geq -1$. Similarly, if $u \in U$ has $u^\top u = 1$ and $t \leq 0$, then 

\[u^\top(tQ_{\lambda^{(1)}} + (1-t)Q_{\lambda^{(2)}})u = tu^\top Q_{\lambda^{(1)}} u - (1-t) = t (u^\top Q_{\lambda^{(1)}} u  +1) - 1 \leq -1,\]
where the inequality holds since $u^\top Q_{\lambda^{(1)}} u \geq -1$. 

The claim then follows from the variational characterization of eigenvalues. 
\end{proof}

We now show that if the hyperbolicity cone $\cP$ is such that $\lambda \in \mathrm{int} \cP$ implies that $Q_\lambda$ has exactly two negative eigenvalues, then there is a unique $\lambda \in \cP \cap \bbS^2$ such that $Q_\lambda$ has a repeated negative eigenvalue. In particular, this will imply that the differential $d_2^{0,n}$ is nontrivial.

\begin{proposition}\label{prop:Unique_neg_rep}
    Suppose that $\Omega^n \not = \emptyset$ and $g$ is smooth and hyperbolic with hyperbolicity cone $\cP$ containing matrices with exactly 2 negative eigenvalues. Then, there is a unique $\lambda \in \cP \cap \bbS^2$ such that $Q_\lambda$ has a repeated negative eigenvalue. 
\end{proposition}

\begin{proof}
First, there can be at most one such $\lambda$. Suppose for the sake of a contradiction that there were $\lambda^{(1)}, \lambda^{(2)} \in \cP$ such that $Q_{\lambda^{(1)}}$ and $Q_{\lambda^{(2)}}$ have repeated negative eigenvalues and $\lambda^{(1)}$ is not a scalar multiple of $\lambda^{(2)}$. By rescaling the $\lambda^{(i)}$ if necessary, we can take the repeated negative eigenvalue to be $-1$. Lemma \ref{lem:Repeated_neg_on_pencil} then implies that for all $t \in \bbR$, the matrix $tQ_{\lambda^{(1)}} + (1-t)Q_{\lambda^{(2)}}$ has at least two negative eigenvalues. This implies that $t\lambda^{(1)} + (1-t)\lambda^{(2)} \in \mathrm{int}(\cP)$ for all $t\in \bbR$ since $\omega \in \partial \cP$ impies that $Q_\omega$ has $n-1$ positive eigenvalues, one negative eigenvalue, and $0$ as an eigenvalue of multiplicity 1. In particular, $g(t\lambda^{(1)} + (1-t)\lambda^{(2)}) \not = 0$ for $t \in \bbR$. On the other hand, $g$ is hyperbolic with respect to $\lambda^{(2)}$, so that $g(\lambda^{(2)} + t(\lambda^{(1)} - \lambda^{(2)})) \in \bbR[t]$ is real rooted. So, the homogenized polynomial $g(s\lambda^{(2)} + t(\lambda^{(1)} - \lambda^{(2)})) \in \bbR[s,t]$ must have a zero of multiplicity $n+1$ at the point $[0:1] \in \bbP^1$, a contradiction with the hypothesis that $g$ is smooth. So, there is at most one $\hat{\lambda} \in \cP$ such that $Q_{\hat{\lambda}}$ has $-1$ as a repeated negative eigenvalue and therefore there is at most one $\lambda = \hat{\lambda}/\|\hat{\lambda}\| \in \cP \cap \bbS^2$ such that $Q_\lambda$ has a repeated negative eigenvalue.

We now show that there is exactly one such $\lambda$. Our strategy is to choose a coordinate system where a matrix with repeated negative eigenvalue exists. Let $\lambda \in \mathrm{int} \cP \cap \bbS^2$ and let $v^{(1)},v^{(2)}, \ldots, v^{(n+1)}$ be an orthonormal basis of eigenvectors of $Q_\lambda$ corresponding to the eigenvalues $\rho_1(Q_\lambda) \geq \rho_2(Q_\lambda) \geq \ldots \geq \rho_{n+1}(Q_\lambda)$. Set $B = \begin{bmatrix}\frac{1}{\sqrt{|\rho_{1}(Q_\lambda)|}}v^{(1)} &  \frac{1}{\sqrt{|\rho_{1}(Q_\lambda)|}}v^{(2)} & \dots & \frac{1}{\sqrt{|\rho_{n+1}(Q_\lambda)|}}v^{(n+1)}\end{bmatrix}$. Then, $x \mapsto Bx$ defines a change of coordinates on $\bbP^{n}$. In these coordinates, the quadratic form $f_\lambda^h$ is represented by the diagonal matrix 

\[ B^\top Q_\lambda B = \begin{bmatrix} 1 & & & &\\ & \ddots & & &\\ & &1 & &\\ & & & -1& \\& & & & -1\end{bmatrix}\]
Next, note that for $\omega \in \bbR^3$, $\det(B^\top Q_\omega B) = \det(B)^2g(\omega)$ so that the zero set and hyperbolicity cone of $g$ are preserved. By the above discussion, $\lambda$ is the unique element of $\cP \cap \bbS^2$ such that $B^\top Q_\lambda B$ has a repeated negative eigenvalue. Theorem \ref{thm:Spectral_Sequence_d2} then implies that the variety $\cV_{\bbR}(B^\top Q_1B, B^\top Q_2 B, B^\top Q_3B)\subseteq \bbP^n$ is empty. Since nonexistence of real points on a variety is preserved under a real coordinate change of $\bbP^n$, this implies that $\cV_\bbR(f_1^h,f_2^h,f_3^h)$ is also be empty. By Theorem \ref{thm:Spectral_Sequence_d2} and the preceding discussion, this implies that there is an odd number of and therefore exactly one $\lambda \in \cP \cap \bbS^2$ such that $Q_\lambda$ has a repeated negative eigenvalue. \end{proof}

We treat the small $n$ cases separately and prove the theorem.

\begin{lemma}\label{lem:n=1_empty_var}
Suppose that $n = 1$ and that $g(\lambda)$ is smooth. Then, $\cV_\bbR(f_1^h,f_2^h,f_3^h) = \emptyset$ if and only if PDLC holds. 
\end{lemma}

\begin{proof}
If PDLC holds, then $X = \cV_\bbR(f_1^h,f_2^h,f_3^h) = \emptyset$. Suppose that PDLC does not hold. Then, by the hypothesis that $g$ is smooth, we have that $Q_\lambda$ has one positive and one negative eigenvalue for all $\lambda \in \bbR^3$ and therefore $H^2(\Omega^1) \cong \bbZ_2$. So, if $(E_r,d_r)$ is the spectral sequence of Theorem \ref{thm:Spectral_SequenceA}, then the $E_2$ page has the form 

\begin{sseqdata}[title = $E_\page$, name = empty var n = 1, cohomological Serre grading, classes = {draw = none }]

\class["\mathbb{Z}_2"](0,1)
\class["0"](0,0)

\class["0"](1,1)
\class["0"](1,0)

\class["0"](2,1)
\class["0"](2,0)

\class["0"](3,1)
\class["\bbZ_2"](3,0)

\end{sseqdata}

\[
\printpage[name = empty var n = 1 , page = 2] 
\]
so that $0 = d_2^{0,1} = d_r^{0,1}$ for all $r$. In particular, this implies that $E_{\infty}^{0,1} \simeq \bbZ_2$ and therefore $H_0(X) = \bigoplus_{i + j = 1}E_{\infty}^{i,j} \not = 0$ and $X \not = \emptyset$. \end{proof}

There is also an elementary proof of Lemma \ref{lem:n=1_empty_var} which follows by observing that when $n = 1$, the space of quadratics has dimension 3, so if the quadratic forms $f_1^h,f_2^h,f_3^h$ are linearly independent, they span the entire space of quadratics and therefore satisfy PDLC. 

\begin{lemma}\label{lem:n=2_empty_var}
Suppose that $n = 2$. If $\cV_\bbR(f_1^h,f_2^h,f_3^h) \not = \emptyset$, then $g(\lambda)$ is not smooth.  \end{lemma}

\begin{proof}

Suppose that $v \in \bbR^3$ is nonzero and $f_i^h(v) = 0$ for $i = 1,2,3$. Then, we have that the vectors $Q_1v, Q_2v,Q_3v$ are all orthogonal to $v$ and therefore $\dim(\spann_\bbR\{Q_1v,Q_2v,Q_3v\}) \leq 2$. So, there is $\alpha \in \bbR^3$ such that $0 = \sum_{i =1}^3 \alpha_iQ_iv = Q_\alpha v$. In particular, $Q_\alpha$ has rank at most two. If $Q_\alpha$ is rank one, then $\alpha$ is a singular point of $\cV_{\bbR}(g)$. Suppose that $Q_\alpha$ has rank two. Let $\cD \subseteq \mathrm{Sym}_3(\bbR)$ be the hypersurface in the space of $3 \times 3$ symmetric real matrices defined by the vanishing of the determinant. Note that $Q_\alpha \in \cD$. By Jacobi's formula, we have that $\nabla \det (Q_\alpha) = \mathrm{adj } Q_\alpha = cvv^\top$ for some constant $c$. This implies that $T_{Q_\alpha}\cD$, the tangent space to $\cD$ at $Q_\alpha$, is the space of quadratic forms which vanish at $v$, as $\langle \mathrm{adj } Q_\alpha, Q \rangle = 0$ implies that $\mathrm{tr}(vv^\top Q) = v^\top Q v = 0$. But then, $Q_i \in T_{Q_\alpha}\cD$ for each $i \in [3]$. So, $\spann_\bbR(\{Q_i \; \vert \; i \in [3]\})$ intersects $\cD$ non-transversely at $Q_\alpha$ and therefore $\alpha$ is a singular point of $g$. \end{proof}

We are now ready to prove Theorem \ref{thm:empty_var_iff_hyperbolic_curve}. 

\EmptyVarIFF*
\begin{proof}
 The cases $n = 1$, $n = 2$, and $n = 3$ are treated by Lemma \ref{lem:n=1_empty_var}, Lemma \ref{lem:n=2_empty_var}, and \cite[Theorem 7.8]{plaumann_quartic_2011}, respectively. The ``only if" statement for $n\geq 4$ is Lemma \ref{lem:Empty_variety_hyperbolic_curve} \cite[Corollary 2]{AgrachevHomologyIntersections}. If PDLC holds, then $X(\{0\},f^h) = \emptyset$ for any $n$.

In the remaining case where PDLC does not hold and $n \geq 4$, we have that $\Omega^n \not = \emptyset$ and $H^1(\Omega^n) \simeq H^2(C\Omega,\Omega^n) \not = 0$. Let $\sigma:B^2 \to C\Omega$ be a representative of the nontrivial class in $H^2(C\Omega,\Omega^n)$. Then, with notation as in the statement of Theorem \ref{thm:Spectral_Sequence_d2}, $\gamma_{1,n}(\bar{f}(\sigma))$ is equal to the number of $\lambda \in \cP \cap \bbS^2$ such that $Q_\lambda$ has a repeated negative eigenvalue. So, by Theorem \ref{thm:Spectral_Sequence_d2} and Proposition \ref{prop:Unique_neg_rep}, there is exactly one such $\lambda$ and we compute that 

\[d_2^{0,n}(1)(\sigma) = 1 \smile \bar{f}^*\gamma_{1,n}(\sigma) = \gamma_{1,n}(\bar{f}(\sigma)) = 1.\] 

So, $d_2^{0,n}:\bbZ_2 \to H^1(\Omega^n)$ is injective and therefore $E_3^{0,n} \simeq E_\infty^{0,n} \simeq 0$. This then implies that $H_0(X) \simeq \bigoplus_{i+j = n}E_\infty^{0,n} \simeq 0$.\end{proof}

Together, the results of this section imply Theorem \ref{thm:intro_thm1}.

\IntroTheoremOne*

\begin{proof}
The statement for $n = 3$ is \cite[Theorem 7.8]{plaumann_quartic_2011}. The remaining statements are Lemmas \ref{lem:Empty_variety_hyperbolic_curve}, \ref{lem:n=1_empty_var}, and \ref{lem:n=2_empty_var} and Theorem \ref{thm:empty_var_iff_hyperbolic_curve}.
\end{proof}

Theorem \ref{thm:empty_var_iff_hyperbolic_curve} places strong structure restrictions on the spectral curve $\cV_\bbR(g)$ in the case that $n \geq 3$ and the real variety is empty. Certificates for emptiness of the projective variety $\cV_\bbR(f_1^h,f_2^h,f_3^h)$ can in turn be leveraged to construct certificates of emptiness for systems of inequalities, as demonstrated in the following proposition. 

\GeneralEmptiness*

\begin{proof}
To simplify notation, we take $\Omega^j = \Omega^j(K)$ for all $j$ and $X = X(K,f^h)$. If $\Omega^{n+1}\neq \emptyset$, then there is $\lambda \in K^\circ \cap S^2$ with $Q_\lambda \succ 0$ and therefore $X = \emptyset$. 

Next, note that $\Omega^n \neq \emptyset$ is a necessary condition for $X = \emptyset$. Otherwise, as in the proof of Lemma \ref{lem:Empty_variety_hyperbolic_curve}, the spectral sequence of Theorem \ref{thm:Spectral_SequenceA} would have that $E_2^{0,n} \simeq E_\infty^{0,n} \simeq \bbZ_2$ is a direct summand of $H_0(X)$ and therefore $H_0(X) \neq 0$. 

Now, suppose that $\Omega^{n+1} = \emptyset$ and $\Omega^n \neq \emptyset$. Note that the hypothesis that $K \not = 0$ implies that $K^\circ \not = \bbR^3$ and therefore $H^2(\Omega^j) = 0$ for all $j$. So, the $E_2$ page of the spectral sequence of Theorem \ref{thm:Spectral_SequenceA} has first four columns

\[
\printpage[name = empty var top four  , page = 2] 
\]

So if $X = \emptyset$, then $d_2^{0,n}:\bbZ_2 \to H^1(\Omega^n)$ is injective and therefore $H^1(\Omega^n) \neq \emptyset$. 

Conversely, suppose that $H^1(\Omega^n) \neq \emptyset$. As in the proof of Theorem \ref{thm:empty_var_iff_hyperbolic_curve}, this implies that the polynomial $g$ is hyperbolic with a hyperbolicity cone $\cP$ such that $\lambda \in \cP$ implies that $Q_\lambda$ has exactly two negative eigenvalues. Moreover, there is exactly one $\lambda \in \cP \cap \bbS^2$ such that $Q_\lambda$ has a repeated negative eigenvalue. But then, the differential $d_2^{0,n}:\bbZ^2 \to H^1(\Omega^n)$ is injective, which implies that $E_3^{0,n} \simeq E_\infty^{0,n} \simeq 0$ and therefore $H_0(X) = 0$.  

\end{proof}

Proposition \ref{prop:General_Emptiness} allows us to certify emptiness of solutions sets of systems of quadratic inequalities using convexity. In particular, if $\cV_\bbR(f_1^h,f_2^h,f_3^h) = \emptyset$ with a smooth hyperbolic spectral curve (as is guaranteed for $n\geq 3$), the condition $\Omega^{n+1}(K) \not = \emptyset$ is equivalent to the statement that there is a hyperbolicity cone $\cP$ of $g$ which intersects $K^\circ$ and that $\lambda\in \cP$ implies $Q_\lambda \succ 0$. Similarly, the condition $H^1(\Omega^n(K)) \not = 0$ is equivalent to the conditions that $\cP\setminus \{0\} \subseteq \mathrm{int}(K^\circ)$ and $Q_\lambda$ has exactly two negative eigenvalues for $\lambda \in \cP$. 

We conclude this section with a computation of the number of connected components of the set $X(K,f^h)$ which will be used in later sections. 

\begin{lemma}\label{lem:Connected_Omega_n}
Suppose that $\cV_\bbR(f_1^h,f_2^h,f_3^h) = \emptyset$ and $K$ is a nonzero polyhedral cone such that $X(K,f^h) \not = \emptyset$ and that $\Omega^n(K)\neq \emptyset$. Then, $H_0(X(K,f^h))\simeq H^0(\Omega^n(K))$.
\end{lemma}

\begin{proof}
Let $(E_r,d_r)$ be the spectral sequence of Theorem \ref{thm:Spectral_SequenceA}. By Proposition \ref{prop:General_Emptiness} and the hypotheses that $X(K,f^h) \neq 0$ and $\Omega^n(K) \neq \emptyset$, the $E_2$ page has the form 

\begin{sseqdata}[title = $E_\page$, no y ticks, name = Connected Omega n, cohomological Serre grading, classes = {draw = none }, xscale = 2]
\begin{scope}[background]
    \node at (-0.6,0) {n-2};
    \node at (-0.6,1) {n-1};
    \node at (-0.6,2) {n};
\end{scope}

\class["\mathbb{Z}_2"](0,2)
\class["0"](0,1)
\class["0"](0,0)

\class["H^0(\Omega^n(K))/\mathbb{Z}_2"](1,1)
\class["H^0(\Omega^{n-1}(K))/\mathbb{Z}_2"](1,0)
\class["0"](1,2)

\class["0"](2,1)
\class["0"](2,0)
\class["0"](2,2)
\end{sseqdata}

\[
\printpage[name = Connected Omega n , page = 2] 
\]
and $E_2^{i,j} \simeq E_\infty^{i,j}$ for all $i,j \in \bbZ$. So, 

\[H_0(X(K,f^h)) \simeq \bbZ_2 \oplus \left(H^0(\Omega^n(K))/\mathbb{Z}_2\right) \simeq H^0(\Omega^n(K)).\] \end{proof}

\section{Finiteness of Aggregations}\label{sec:Finiteness}

In this section, we develop a strategy to reduce sets of aggregations to finite subsets under the assumption that $\cV_\bbR(f_1^h,f_2^h,f_3^h) = \emptyset$ and the spectral curve is smooth and hyperbolic. In particular, we show that under these assumptions, there is a finite subset $\Lambda_1 \subseteq \Lambda$ such that $\bigcap_{\lambda \in \Lambda}S_\lambda = \bigcap_{\lambda \in \Lambda_1}S_\lambda$. In order to effectively apply the techniques presented in Section \ref{sec:no_solns}, we interpret the necessity of a given aggregation in terms of the emptiness of a certain system of quadratic equations and inequalities. 

\UnnecessaryAggstwo*

\begin{proof}
Note that $X(-\bbR_+^k,f_k^h) \subseteq X(-\bbR_+^{k+1},f_{k+1}^h)$. It remains to show the reverse inclusion. If $f_{\lambda^{(k+1)}}^h([x]) \not = 0$ for all $[x] \in X(-\bbR^{k+1}_+, f_{k+1}^h)$, then $f_{\lambda^{(k+1)}}^h$ has constant nonzero sign on each connected component of $X(-\bbR^{k},f^h_k)$. Since $X(-\bbR_+^{k+1},f_{k+1}^h)$ has the same number of connected components as $X(-\bbR^{k}_+,f^h_k)$, this implies that $f_{\lambda^{(k+1)}}^h$ is negative on each component of $X(-\bbR^k_+,f^h_k)$. So, the equality holds. 
\end{proof}

\begin{remark}\label{rem:UnnecessaryAggsProjAff}
While Lemma \ref{lem:UnnecessaryAggstwo} is stated projectively, its conclusion implies the affine statement $\bigcap_{j = 1}^{k} S_{\lambda^{(j)}} = \bigcap_{j = 1}^{k+1}S_{\lambda^{(j)}}$. To see this, note that if $x \in \bigcap_{j = 1}^{k} S_{\lambda^{(j)}}$ but $x \not \in S_{\lambda^{(k+1)}}$, then $f_{\lambda^{(k+1)}}(x) > 0$ so that $f_{\lambda^{(k+1)}}^h((x,1)) > 0$ and therefore $[(x,1)] \in X(-\bbR^k_+,f_k^h) \setminus X(-\bbR^{k+1}_+, f_{k+1}^h)$.
\end{remark}

The hypotheses in Lemma \ref{lem:UnnecessaryAggstwo} can be checked by applying the tools in Section \ref{sec:no_solns} using the cone $K = -\bbR^{k}_+ \times \{0\}$. 

We apply the construction in Lemma \ref{lem:UnnecessaryAggstwo} specifically to the set of unit length generators of extreme rays of $\Lambda$ which lie on a proper face of $\bbR^3_+$,

\[\Lambda' = \{\lambda \in \mathrm{ex}(\Lambda) \; \vert \;  |\supp(\lambda)| \leq 2\}.\]
Note that $\Lambda'$ has at most 6 elements since each facet of $\bbR^3_+$ can contain at most two extreme rays of $\Lambda$. The aggregations in $\Lambda'$ suffice to describe $\conv(\mathrm{int}(S))$ when the $Q_i$ satisfy PDLC by \cite[Theorem 2.17]{blekherman_aggregations_2022}. We have the following description of elements of $\Lambda'$. 

\begin{proposition}\label{prop:Lambda_prime_elts}
    If $\lambda \in \Lambda'$, then either $\lambda$ is a standard basis vector of $\bbR^3$ or $g(\lambda) = 0$. 
\end{proposition}

\begin{proof}
If $\lambda$ is not a standard basis vector, then $\lambda = \lambda_ie_i + \lambda_je_j$ for some $i\not = j \in [3]$ and $\lambda_i,\lambda_j > 0$. Since $\lambda \in \Lambda'$, if $g(\lambda) \not = 0$, then $Q_\lambda$ has exactly $n$ positive and one negative eigenvalue. So, for $\epsilon > 0$ sufficiently small, $\lambda - \epsilon(e_i + e_j)$ and $\lambda +\epsilon(e_i + e_j)$ are elements of $\bbR^3_+$ and $Q_{\lambda\pm \epsilon(e_i + e_j)}$ also has exactly $n$ positive and one negative eigenvalue. But then, $\lambda$ does not span an extreme ray of $\Lambda$, a contradiction.\end{proof}

We now work towards applying Lemma \ref{lem:UnnecessaryAggstwo} to the case that $\lambda^{(1)},\lambda^{(2)},\ldots, \lambda^{(k)}$ is an enumeration of $\Lambda'$ and $\lambda^{(k+1)} \in \Lambda \setminus \cone(\Lambda')$. First, we characterize faces of the relevant cone. 

\begin{proposition}\label{prop:Faces_ATK}
Suppose that $\lambda^{(1)},\lambda^{(2)},\ldots, \lambda^{(k)}$ is an enumeration of $\Lambda'$ and $\lambda^{(k+1)} \in \Lambda \setminus \cone(\Lambda')$. Let $A \in \bbR^{(k+1) \times 3}$ have rows $\left(\lambda^{(i)}\right)^\top$. If $F = \{t\lambda^{(k+1)} + s \lambda^{(i)} \;\vert \; i\in [k], t\in \bbR, s\in \bbR_+\}$ is a face of $A^\top(\bbR^k_+ \times \bbR)$ then $g(\lambda^{(i)}) = 0$.
\end{proposition}

\begin{proof}
Note that if $\lambda^{(k+1)} \in \Lambda \setminus \cone(\Lambda')$, then $\lambda^{(k+1)}$ has strictly positive entries. Let $v \in \bbR^3$ support the face $F$ so that $v^\top \lambda^{(k+1)} = v^\top \lambda^{(i)} = 0$ and $v^\top \lambda^{(j)} > 0$ for $j \in [k]\setminus \{i\}$. Suppose for the sake of a contradiction that $g(\lambda^{(i)}) \not = 0$. By Proposition \ref{prop:Lambda_prime_elts}, this implies that $\lambda^{(i)}$ is a standard basis vector of $\bbR^3$. Relabeling if necessary, take $\lambda^{(i)} = e_1$. 

Next, note that since $\lambda^{(i)} = e_1$ and $g(\lambda^{(i)}) \not = 0$, the matrix $Q_{\lambda^{(i)}} = Q_1$ has exactly $n$ positive and one negative eigenvalue. Therefore, either $e_2 \in \Lambda'$ or there is $t \in [0,1)$ such that $te_1 + (1-t)e_2 \in \Lambda'$. But then $v^\top e_2 > 0$. Similarly, $v^\top e_3 >0$. Since $\lambda^{(k+1)}$ was assumed to have strictly positive entries, $v^\top \lambda^{(k+1)} > 0$, a contradiction with the construction of $v$. 
\end{proof}

The following lemmas will also be used in the proof of Theorem \ref{thm:FiniteAggs}. Both are concerned with the behavior of aggregations which lie on the oval of depth $\lfloor \frac{n+1}{2} \rfloor - 1$. Let $C$ be the part of the affine cone over the oval of depth $\lfloor \frac{n+1}{2}\rfloor - 1$ where if $\omega \in C$ then $Q_\omega$ has $n-1$ positive eigenvalues, one negative eigenvalue, and 0 as an eigenvalue of multiplicity one. Lemma \ref{lem:Agg_Not_in_Oval} establishes a root-counting property. Lemma \ref{lem:Second_oval_intersection} describes the intersection of $C$ with $\bbR^3_+$. 

\begin{lemma}\label{lem:Agg_Not_in_Oval} 
Let $\lambda \in \Lambda \cap C$. Then, if $e \in \cP$, there is exactly one root $t^*$ of $g(te + (1-t)\lambda)$ for $t \in (0,1]$ and $t^*e + (1-t^*)\lambda$ lies on $\partial \cP$.
\end{lemma}

\begin{proof}
First, we can reduce to the case that $g$ is a definite representation of the spectral curve. Indeed, such a definite representation always exists \cite{HeltonVinnikovLMIsets} and the intersection points of the curve with the line $L = \{te + (1-t)\lambda \; \vert \; t \in \bbR\}$ do not depend on the representation of the curve. 

Note that there is at least one such root since $e \in \cP$. Now, there are at least $n-1$ roots in $(-\infty,0) \cup (1,\infty)$ since for $t$ sufficiently large, $Q_{\lambda + t(e - \lambda)}$ and $Q_{\lambda - t(e-\lambda)}$ have opposite signature. Since $0$ is a root and there is a root in $(0,1]$, there can be no other roots, as $g(te + (1-t)\lambda)$ has degree $n+1$.

\end{proof}

\begin{lemma}\label{lem:Second_oval_intersection}
 Suppose that $C \cap \bbR^3_+ \not = \emptyset$. Then, each connected component of $C \cap \bbR^3_+$ intersects a proper face of $\bbR^3_+$. Moreover, if a connected component of $C \cap \bbR^3_+$ intersects $\mathrm{int}(\bbR^3_+)$, then there exist $\lambda^{(1)}\not = \lambda^{(2)}$ in the intersection $C \cap \bbR^3_+$ with $|\supp(\lambda^{(i)})| \leq 2$. 
\end{lemma}

\begin{proof}
Note that the interior of $C$ consists of $\lambda$ such that $Q_\lambda$ has $n$ positive and 1 negative eigenvalue. If a connected component $C_1$ of $C \cap \bbR^3_+$ does not intersect a proper face of $\bbR^3_+$, then $C_1$ is contained entirely in $\mathrm{int}(\bbR^3_+)$. This implies that the image of $C_1$ in $\bbP^2$ is an oval of the spectral curve of depth $\lfloor \frac{n+1}{2}\rfloor - 1$. Since there is only one such component, we have that $C_1 = C$. Now, the hyperbolicity cone $\cP$ is contained in the interior of the region bounded by $C$ and therefore if $C$ does not intersect a proper face of $\bbR^3_+$, we have that $\cP \subseteq \mathrm{int}(\bbR^3_+)$. By Proposition \ref{prop:General_Emptiness}, this implies that $S = \emptyset$, a contradiction. 

The preceding argument shows that $\bbR^3_+$ cannot contain the entirety of the region bounded by $C$. If a connected component $C_1$ of $C \cap \bbR^3_+$ contains an element with strictly positive entries and has a unique element $\lambda$ with $\left|\supp(\lambda)\right| \leq 2$, then the entirety of the region bounded by $C_1$ is contained in $\bbR^3_+$. In particular, we have that $C_1 = C$ and applying the preceding argument gives the desired contradiction. 
\end{proof}

We now make the following calculation which will be used in the proof of Theorem \ref{thm:FiniteAggs}.

\begin{proposition}\label{prop:positive_agg_doesnt_contribute}
Assume that $C \cap \bbR^3_+ \not = \emptyset$. Suppose further that there are $e \in \cP$ and $\lambda^{(1)},\lambda^{(2)} \in \Lambda \cap C$ such that $t\mu + (1-t)e = \lambda$ for some $\lambda \in \cone(\lambda^{(1)},\lambda^{(2)})$ and $0\leq t\leq 1$. Set $A^\top = \begin{bmatrix} \lambda^{(1)} & \lambda^{(2)} & \mu \end{bmatrix}$. Then,

\[\cP \subseteq \mathrm{int}\left(A^\top\left(\bbR^2_+\times \bbR \right)\right).\]
\end{proposition}

\begin{proof}
We proceed in two steps. First, we show that $\cP \cap \mathrm{int}\left(A^\top\left(\bbR^2_+ \times \bbR\right)\right) \neq \emptyset$. Suppose for the sake of a contradiction that $A^\top(\bbR_+^2 \times \bbR) \cap \cP = \emptyset$. Let $v$ be normal to a separating hyperplane oriented so that $v^\top y > 0$ for all $y \in \cP$ and $v^\top w \leq 0$ for all $w \in A^\top(\bbR_+^2 \times \bbR)$. Then $v^\top \mu < 0$ by the expression $t\mu = \lambda - (1-t)e$. But then, if $(c_1,c_2,d) \in \bbR^2_+ \times \bbR$, we have that $v^\top \left(d\mu + \sum_{j = 1}^{2}c_j\lambda^{(j)}\right)$ can take both positive and negative values by considering large enough positive and negative values of $d \in \bbR$. So, $v$ cannot be the normal to a separating hyperplane, the desired contradiction. 

We now show that $\{t\mu + s\lambda^{(i)} \; \vert \; t \in \bbR, s \in \bbR_+\}$ does not intersect $\cP$. This follows from Lemma \ref{lem:Agg_Not_in_Oval} since if this intersection contained some element $e$, then there would be two roots of $g$ restricted to the line segment between $\mu$ and $e$. 
\end{proof}

The inclusion of cones in the conclusion of Proposition \ref{prop:positive_agg_doesnt_contribute} implies that either $\Omega^{n+1}((A^\top(\bbR^2_+ \times \bbR))^\circ) \not = \emptyset$ or that the set $\Omega^n((A^\top(\bbR^2_+ \times \bbR))^\circ)$ has nontrivial first cohomology group. This calculation is leveraged to prove Theorem \ref{thm:FiniteAggs}.

\FiniteAggs*

\begin{proof}

Enumerate $\Lambda' = \{\lambda^{(1)},\lambda^{(2)},\ldots, \lambda^{(k)}\}$ and let $\lambda^{(k+1)} \in \Lambda \setminus \cone(\Lambda')$. We will show that $\lambda^{(k+1)}$ is not necessary to describe the set defined by all aggregations by showing that $\bigcap_{i = 1}^k S_{\lambda^{(i)}} = \bigcap_{i = 1}^{k+1}S_{\lambda^{(i)}}$. Note that if $g(\lambda^{(k+1)}) \not= 0$, then we can write $\lambda^{(k+1)}$ as a conical combination of an element of $\cone(\Lambda ')$ and an element of $C$, where $C$ is the portion of the affine cone over the oval of depth $\lfloor \frac{n+1}{2}\rfloor - 1$ which contains matrices with one negative eigenvalue. So, it suffices to take $\lambda^{(k+1)} \in C$. By the hypothesis that $S \not = \emptyset$, the hyperbolicity cone $\cP \not \subseteq \bbR^3_+$. This implies that there is $e \in \cP$ and $\lambda^{(i)},\lambda^{(j)}\in \Lambda$ such that $t\lambda^{(k)} + (1-t)e = \lambda$ for some $\lambda \in \cone(\lambda^{(i)},\lambda^{(j)})$ and $0\leq t\leq 1$. Since $\lambda^{(k+1)} \not \in \cone(\Lambda')$, we can take $\lambda^{(i)},\lambda^{(j)} \in C$. 

Set $K = \cone(\lambda^{(i)},\lambda^{(j)},\mu,-\mu)$. By Proposition \ref{prop:positive_agg_doesnt_contribute}, $\cP \subseteq \mathrm{int}(K^\circ)$ and therefore either $H^1(\Omega^{n}(K^\circ)) \not = 0$ or $\Omega^{n+1}(K^\circ) \neq \emptyset$. By Proposition \ref{prop:General_Emptiness}, this implies that $X(K^\circ,f^h) = \emptyset$, i.e. $\cV_\bbR(f_{\lambda^{(k+1)}}^h) \cap X(-\bbR^2_+, (f_{\lambda^{(i)}}^h,f_{\lambda^{(j)}}^h)) = \emptyset$. This in turn implies that $\cV_{\bbR}(f_{\lambda^{(k+1)}}^h) \cap X(\cone(\Lambda')^\circ, f^h) = \emptyset$. By the hypothesis that $\Omega^n(\cone(\Lambda')^\circ)$ is contractible and Lemma \ref{lem:UnnecessaryAggstwo}, this implies that $\lambda^{(k+1)}$ is unnecessary, so that $\bigcap_{i =  1}^{k}S_{\lambda^{(i)}} = \bigcap_{i = 1}^{k+1}S_{\lambda^{(i)}}$.
\end{proof}

Figure \ref{fig:Cohomology} illustrates the argument of the proof of Theorem \ref{thm:FiniteAggs}. Again, we take the system of quadratics from Example \ref{ex:ch_not_aggs} so that the curve $\cV_{\bbR}(g)$ is hyperbolic but $g$ is not a definite representation. The aggregation $\lambda^{(k+1)}$ which lies in the interior of $\bbR^3$ does not contribute to the intersection of all aggregations of correct signature, which is given by $\bigcap_{\lambda \in \Lambda'}S_{\lambda}$. 

\begin{figure}[t]
    \centering
    \includegraphics[width = 0.5\textwidth]{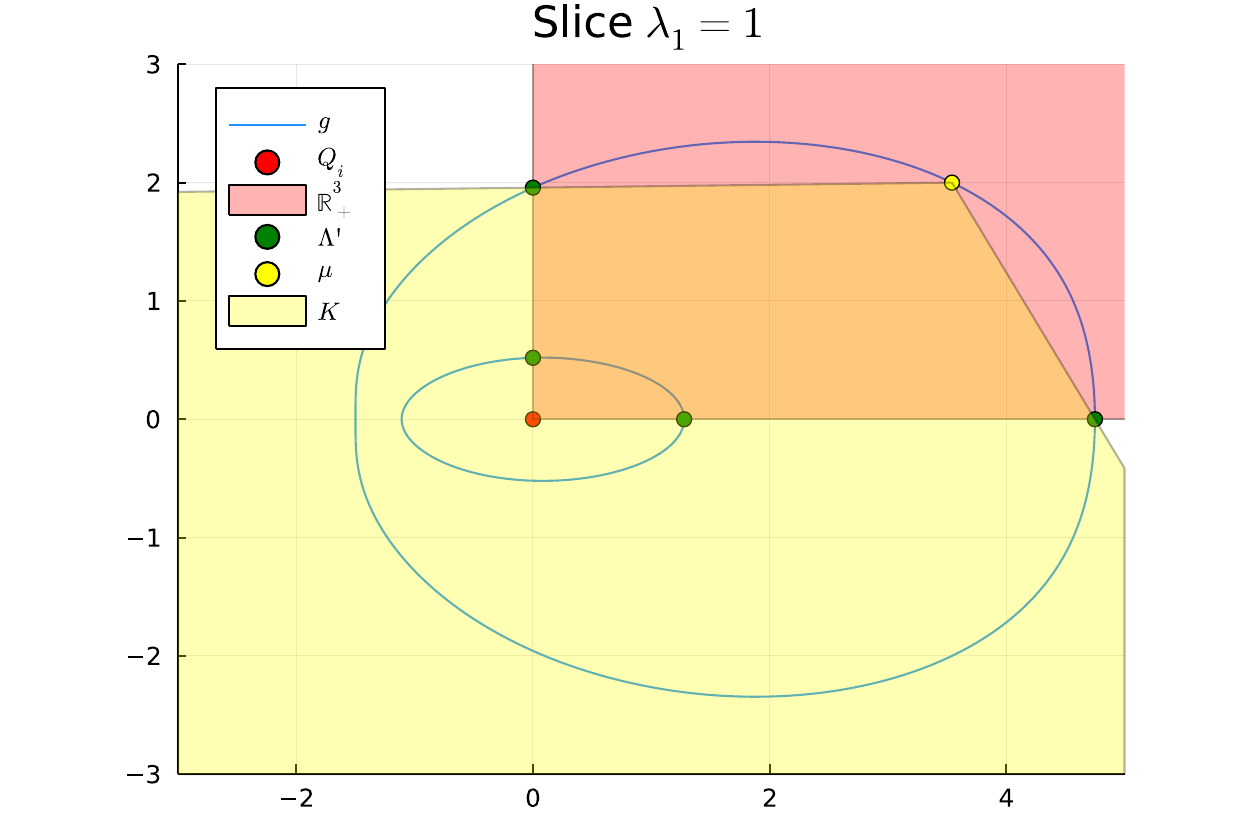}
    \caption{An illustration of Theorem \ref{thm:FiniteAggs}. The set described by aggregations of the correct signature is $\bigcap_{\lambda \in \Lambda'}S_\lambda$. The aggregation $\mu$ does not contribute to this description. This is certified by the fact that the cone $K = A^\top (\bbR^4_+\times \bbR)$ contains the hyperbolicity cone of $g$ in its interior. In cohomological terms, $H^1(\Omega^n(K^\circ)) = \bbZ_2$.}
    \label{fig:Cohomology}
\end{figure}

As shown in Example \ref{ex:positive_agg_needed}, the set defined by the intersection of all permissible aggregations is not necessarily defined by the intersection of aggregations in $\Lambda'$ when the assumption on the contractibility of $\Omega^n(\cone(\Lambda')^\circ)$ is not met. However, we are able to develop a bound on the number of aggregations with positive entries needed when PDLC is not satisfied and the set defined by all permissible aggregations is not connected. First, we observe that adding an aggregation with strictly positive entries to $\Lambda'$ cannot make the number of connected components increase. 

\begin{lemma}\label{lem:positive_agg_components}
Set $K_0 = \cone(\Lambda')$ and let $\mu^{(1)}, \mu^{(2)}, \ldots, \mu^{(N)} \in \Lambda$ have strictly positive entries. For each $1\leq j\leq N$, set $K_j = \cone(\Lambda' \cup \bigcup_{i = 1}^{j}\{\mu^{(i)}\})$. Then, every connected component of $\Omega^{n}(K_j^\circ)$ intersects a component of $\Omega^{n}(K_0^\circ)$. Moreover, 

\[1 \leq \dim H^{0}(\Omega^n(K_N^\circ)) \leq \dim H^{0}(\Omega^n(K_{N-1}^\circ)) \leq  \dots \leq \dim H^0(\Omega^n(K_1^\circ)) \leq \dim H^0(\Omega^n(K_0^\circ)).\] 
\end{lemma}

\begin{proof}
We first show that there are no connected components of $\Omega^n(K_j^\circ)$ where the points contain all positive entries. Suppose for the sake of a contradiction that there was such a connected component. Then, its boundary would form an oval of $\cV_\bbR(g)$ whose interior contains matrices with $n$ positive eigenvalues. If $\cV_\bbR(f_1^h,f_2^h,f_3^h) = \emptyset$, then $\cV_\bbR(g)$ is hyperbolic and the only oval which contains matrices with $n$ positive eigenvalues in its interior is the oval of depth $\lfloor \frac{n+1}{2}\rfloor - 1$. So, if $K_j$ contains this oval, it contains the oval of maximum depth and therefore the hyperbolicity cone of $g$ which consists of either positive definite matrices or matrices with exactly two negative eigenvalues. By Proposition \ref{prop:General_Emptiness}, this certifies that $X(K_j^\circ,f^h) = \emptyset$. This is a contradiction since $S \not = \emptyset$ and therefore $X(-\bbR^3_+,f^h) \not = \emptyset$ but $X(-\bbR^3_+,f^h) \subseteq X(K_j^\circ,f^h)$.

This implies that every connected component of $\Omega^n(K_j^\circ)$ has nonempty intersection with $\Omega^n(K_0^\circ)$. Indeed, there is $\lambda$ in this component with $\lambda_i = 0$ for some $i \in [3]$. Therefore either $\lambda \in \Lambda'$ or $\lambda$ is a conical combination of two elements of $\Lambda'$. Note that since $\Omega^n(K_0^\circ) \subseteq \Omega^n(K_{j-1}^\circ) \subseteq \Omega^n(K_j^\circ)$ for all $j$, this implies that  each connected component of $\Omega^n(K_j^\circ)$ intersects $\Omega^{n}(K_{j-1}^\circ)$ .

The last claim follows since $\Omega^n(K_{j-1}^\circ) \subseteq \Omega^n(K_j^\circ)$ for each $j$. So, if each connected component of $\Omega^n(K_j^\circ)$ intersects $\Omega^n(K_{j-1}^\circ)$, then there are at least as many connected components of $\Omega^n(K_{j-1}^\circ)$ as connected components of $\Omega^n(K_j^\circ)$. Finally, $\dim H^0(\Omega^n(K^\circ_N))\not=0$ by Theorem \ref{thm:Spectral_SequenceA} since $\emptyset \not = X(-\bbR^3_+,f^h)\subseteq X(K_N^\circ,f^h)$ and therefore $\dim H_0(X(K_N^\circ,f^h)) \geq 1$.
\end{proof}

We can now prove Theorem \ref{thm:NoPDLCDisconnectBound}. 

\NoPDLCDisconnectBound*

\begin{proof}
By Lemma \ref{lem:UnnecessaryAggstwo} and Proposition \ref{prop:positive_agg_doesnt_contribute}, given some finite list of aggregations $\lambda^{(1)},\lambda^{(2)},\ldots, \lambda^{(k)}$ and the cone $K_0 = \cone(\{\lambda^{(i)}\; \vert \; i \in [k]\})$, a sequence of aggregations $\mu^{(1)},\mu^{(2)},\ldots, \mu^{(N)}\in \Lambda$, each with positive entries, contains only necessary elements only if $\dim H^0(\Omega^n(K_j^\circ)) \not = \dim H^0(\Omega^n(K_{j-1}^\circ))$, where $K_j = \cone(K_0\cup \bigcup_{\ell = 1}^j\{\mu^{(\ell)}\})$. By Lemma \ref{lem:positive_agg_components}, this only happens if  $\dim H^0(\Omega^n(K_{j})^\circ) < \dim H^0(\Omega^n(K_{j-1}^\circ))$. In particular, the maximal length $N$ of a sequence of aggregations with positive entries such that $\dim H^0(\Omega^n(K_{j}^\circ))<\dim H^0(\Omega^n(K_{j-1}^\circ))$ for all $1\leq j\leq N$, starting with $K_0 = \cone(\Lambda')$, is at most $\dim H^0(\Omega^n(\cone(\Lambda')^\circ)) - 1$.
\end{proof}

Together, Theorems \ref{thm:FiniteAggs} and \ref{thm:NoPDLCDisconnectBound} imply Theorem \ref{thm:FinitePermissibleAggs}.

\FinitePermissibleAggs*

\begin{proof}
If $\Omega^n(\cone(\Lambda')^\circ)$ is contractible, then $\Lambda_1 = \Lambda'$ by Theorem \ref{thm:FiniteAggs}. Otherwise, the result is Theorem \ref{thm:NoPDLCDisconnectBound}.
\end{proof}

\section{A Sufficient Condition for Obtaining the Convex Hull via Aggregations}\label{sec:SufficientCondition}

In this section, we present proofs of the results stated in Section \ref{sec:suff_agg}. Throughout this section, we assume that $\cV_\bbR(f_1^h,f_2^h,f_3^h) = \emptyset$. Our approach relies on certifying that $S^h \cap H = 0$ using positive definite aggregations of the restrictions $Q_i\vert_H$.

We first prove a variant of {\cite[Lemma 5.4]{dey_obtaining_2022}} that is compatible with the closed inequality and projective setting. 

\begin{lemma}\label{lem:valid_hyperplanes}
Suppose that $S^h \cap \{(x,x_{n+1}) \in \bbR^{n+1} \; \vert\; x_{n+1} = 0\} = \{0\}$. If $\alpha \in \bbR^{n}$ and $\beta \in \bbR$ are such that $\alpha^\top x < \beta$ for all $x \in S$ and if $H = \{(u,u_{n+1}) \in \bbR^{n+1} \; \vert \; \alpha^\top u = \beta u_{n+1}\}$, then $S^h \cap H = 0$. 
\end{lemma}

\begin{proof}
If $0 \not = (\hat{u},\hat{u}_{n+1}) \in S^h \cap H$, then $\frac{\hat{u}}{|\hat{u}_{n+1}|} \in S$ and $\alpha^\top \left(\frac{\hat{u}}{\hat{u}_{n+1}}\right) = \beta$.
\end{proof}

We now record a variant of the proof strategy of \cite[Theorem 2.4]{dey_obtaining_2022} and \cite[Theorem 2.8]{blekherman_aggregations_2022}.

\HyperplaneAggs*

\begin{proof}
Suppose that $y \not \in \cconv(S)$. Then, there is $\alpha\in \bbR^n$ such that $\alpha^\top x < \alpha^\top y$ for all $x \in S$. Let $\beta = \alpha^\top y$ and $H = \{(x,x_{n+1}) \; \vert \; \alpha^\top x = \beta x_{n+1}\}$. If property \eqref{eq:emptiness_cert} holds, then there is $\lambda \in \bbR^3_+$ such that $Q_\lambda \vert_H \succ 0$. By the Cauchy Interlacing Theorem and the hypothesis that $\mathrm{int}(S) \not = \emptyset$, it follows that $Q_\lambda$ has exactly one negative eigenvalue. Since $Q_\lambda \vert_H \succ 0$, this in turn implies that either $S_\lambda$ is convex or that $S_\lambda$ consists of two disjoint convex connected components separated by the affine hyperplane $\{x \in \bbR^n \; \vert \; \alpha^\top x = \beta\}$. Since $\alpha^\top x < \beta$ is valid on $S$, this implies that $\cconv(S)$ is contained in a single convex connected component of $S_\lambda$ and therefore $\lambda$ is a good aggregation which certifies that $y \not \in \cconv(S)$.  \end{proof}

In light of Proposition \ref{prop:HyperplaneAggs}, any condition which ensures that the emptiness of $S^h\cap H$ is certified by positive definite aggregations of the $Q_i\vert_H$ also ensures that $\cconv(S)$ has a representation as the intersection of good aggregations.

We first restrict our attention to the case that the polynomials $g(\lambda) = \det(Q_\lambda)$ and $g_H(\lambda) = \det(Q_\lambda \vert_H)$ are both smooth and hyperbolic. The first step is to understand the relationship between the curves $\cV_\bbR(g)$ and $\cV_\bbR(g_H)$. In the case that $Q_1,Q_2,Q_3$ satisfy PDLC, then $g_H$ interlaces $g$ \cite{dym_computing_2012}. In the case that $\cV_\bbR(f_1^h,f_2^h,f_3^h) = \emptyset$ but the $Q_i$ do not satisfy PDLC, the relationship is more subtle.  If $g$ is hyperbolic, we set $\cP$ to be the hyperbolicity cone of $g$ which contains either positive definite matrices or matrices with exactly two negative eigenvalues.

If $\cV_{\bbR}(f_1^h,f_2^h,f_3^h) = \emptyset$, then $\cV_{\bbR}(f_1^h,f_2^h,f_3^h) \cap H = \emptyset$. Theorem \ref{thm:empty_var_iff_hyperbolic_curve} then implies that a hyperbolicity cone $\cP_H$ of $g_H$ contains either positive definite matrices or matrices with exactly two negative eigenvalues. We apply interlacing properties of eigenvalues to understand how the ovals of $\cV_{\bbR}(g)$ and $\cV_\bbR(g_H)$ interact by restricting to a pencil (see also \cite{thompson1991pencils,thompson1976characteristic}).

\begin{lemma}\label{lem:zeros_odd} 
Suppose that $\lambda^{(1)}, \lambda^{(2)} \in \bbR^3$ are such that $g(\lambda^{(i)}) = 0$ for $i = 1,2$ and that $g(t\lambda^{(1)} + (1-t)\lambda^{(2)}) \not = 0$ for $t \in (0,1)$. Suppose further that $g_H(\lambda^{(i)}) \not= 0$ for $i = 1,2$. The number of zeros of $g_H(t\lambda^{(1)} + (1-t)\lambda^{(2)})$ for $t \in (0,1)$ is even if $Q_{\lambda^{(1)}}$ and $Q_{\lambda^{(2)}}$ have the same signature and odd otherwise.
\end{lemma}

\begin{proof}
Let $r \in \bbZ$ be such that $Q_{\lambda^{(1)}}$ has $r$ positive eigenvalues, $n-r$ negative eigenvalues, and 0 as an eigenvalue of multiplicity one. Then, $Q_{\lambda^{(1)}}\vert_H$ also has $r$ positive eigenvalues and $n-r$ negative eigenvalues by the Cauchy Interlacing Theorem. 

If $Q_{\lambda^{(1)}}$ and $Q_{\lambda^{(2)}}$ have the same signature, then $Q_{\lambda^{(2)}}\vert_H$ also has $r$ positive and $n-r$ negative eigenvalues. In particular, $g_H(\lambda^{(1)})$ and $g_H(\lambda^{(2)})$ have the same sign which implies that $g_H(t\lambda^{(1)} + (1-t)\lambda^{(2)}) = 0$ for an even number of $t \in (0,1)$. 

On the other hand, if $Q_{\lambda^{(2)}}$ has $r+1$ positive eigenvalues, $n-r-1$ negative eigenvalues, and 0 as an eigenvalue of multiplicity one, then $Q_{\lambda^{(2)}}\vert_H$ has $r+1$ positive eigenvalues and $n-r-1$ negative eigenvalues. In particular, $g_H(\lambda^{(1)})$ and $g_H(\lambda^{(2)})$ have different signs and therefore $g_H(t\lambda^{(1)} + (1-t)\lambda^{(2)})$ for an odd number of $t\in (0,1)$. The case where $Q_{\lambda^{(2)}}$ has $r-1$ positive and $n-r+1$ negative eigenvalues is similar.
\end{proof}

Lemma \ref{lem:zeros_odd} reduces the possible interactions between $\cV_\bbR(g)$ and $\cV_\bbR(g_H)$. In order to prove Theorem \ref{thm:HyperbolicityContainment}, we use the following spectral sequence from \cite{agrachev_systems_2012} which computes relative homology groups of hyperplane sections of the solution set to a system of quadratic inequalities.

\begin{theorem}[{\cite[Theorem D]{agrachev_systems_2012}}]\label{thm:Relative_Spectral_Sequence}
    Fix a polyhedral cone $K \subseteq \bbR^m$, a homogenous quadratic map $p:\bbR^{n+1} \to \bbR^m$, and a hyperplane $\bar{H}\subseteq \bbP^n$. There is a first quadrant cohomology spectral sequence $(G_r,d_r)$ converging to $H_{n-*}(X(K,p),X(K,p)\cap \bar{H})$ with 

    \[G_2^{i,j} = H^i(\Omega^j_H(K), \Omega^{j+1}(K)) \text{ for } j > 0, \quad G_2^{i,0} = H^i(C\Omega(K),\Omega^1(K)).\]
\end{theorem}

Applying Theorem \ref{thm:Relative_Spectral_Sequence} to hyperplane sections of an empty variety gives the following

\begin{corollary}\label{cor:restrict_cohomology_cert}
    Fix a hyperplane $H \subseteq \bbR^{n+1}$. Suppose that $\cV_\bbR(f_1^h,f_2^h,f_3^h) = \emptyset$, and that the polynomial $g$ is smooth and hyperbolic, and that $\cP$ does not contain a positive definite matrix. Then, $H^1(\Omega^{n-1}_H(\{0\}),\Omega^n(\{0\})) = 0$. That is, every noncontractible loop in $\Omega^{n-1}_H(\{0\})$ can be deformed to a noncontractible loop in $\Omega^{n}(\{0\})$. 
\end{corollary}

\begin{proof}
Let $(G_r,d_r)$ be the spectral sequence of Theorem \ref{thm:Relative_Spectral_Sequence}. Set $X = \cV_{\bbR}(f_1^h,f_2^h,f_3^h)$, $\Omega^j_H = \Omega^j_H(\{0\})$, and $\Omega^j = \Omega^j(\{0\})$ for notational simplicity. 

First suppose that $n = 2$. We then have that $G_\infty^{1,1} \simeq \ker(d_2^{1,1}:G_2^{1,1} \to G_2^{3,0})$. Since $X$ and $X \cap \bar{H}$ are empty,

\[0 = H_0(X,X\cap \bar{H}) \simeq G_{\infty}^{2,0}\oplus G_{\infty}^{1,1} \oplus G_{\infty}^{0,2}.\]
Therefore, $G_{\infty}^{1,1} \simeq 0$ and $d_2^{1,1}$ is injective. But $G_2^{3,0} = H^3(C\Omega, \Omega^1) = H^3(B^3,\bbS^2) = 0$ and therefore $G_2^{1,1} = H^1(\Omega_H^{1},\Omega^2) = 0$.

Now suppose that $n \geq 3$. We first show that $G^{i,j}_2 = 0$ when $i\geq 3$ and $j \geq 1$. From the long exact sequence of the pair $(\Omega^j_H,\Omega^{j+1})$, we get the following exact sequence

\[\cdots \to H^{i-1}(\Omega^{j+1}) \to H^{i}(\Omega_H^{j},\Omega^{j+1}) \to H^{i}(\Omega_H^{j}) = 0 \to \cdots,\]
where $H^i(\Omega_H^{j}) = 0$ for $i \geq 3$. So it suffices to show that $H^{i-1}(\Omega^{j+1}) = 0$ for $i\geq 3$ and $j \geq 1$. By Theorem \ref{thm:empty_var_iff_hyperbolic_curve} and the hypothesis that $\cP$ contains matrices with exactly two negative eigenvalues, we compute that if $n\geq 4$, then $\Omega^{j+1}$ is homotopy equivalent to the union of two points when $j = 1$, homotopy equivalent to a point when $2\leq j \leq n-2$, homotopy equivalent to $\bbS^1$ when $j = n-1$, and empty when $j \geq n$. If $n = 3$, $\Omega^{1+1}$ is homotopy equivalent to the union of two points, $\Omega^{2+1}$ is homotopy equivalent to $\bbS^1$, and $\Omega^{3+1} = \emptyset$.

In all cases, we see that $H^{i-1}(\Omega^{j+1}) = 0$ when $i \geq 3$ so that 

\[0 \to H^i(\Omega^i_H,\Omega^{j+1}) \to 0\]
is exact. So, $G^{i,j}_2 = H^i(\Omega^i_H,\Omega^{j+1}) = 0$ when $i\geq 3$ and $j \geq 1$. 

Next, since $\spann\{Q_1,Q_2,Q_3\}$ does not contain a positive semidefinite matrix, it does not contain a negative semidefinite matrix so that $\Omega^1 = \bbS^2$. This implies that $G^{i,0} = H^i(C\Omega, \Omega^1) = H^i(B^3,\bbS^2) = 0$ for $i \geq 4$. 

The vanishing of $G_2^{i,j}$ for $i\geq 3, j \geq 1$ and $i \geq 4$ and the hypothesis that $n \geq 3$ imply that $H^1(\Omega^{n-1}_H,\Omega^n) \simeq G_2^{1,n-1} \simeq G_\infty^{1,n-1}$. Since $X$ and $X \cap \bar{H}$ are empty, and since 

\[0 = H_0(X,X\cap \bar{H}) \simeq \bigoplus_{j = 0}^{n} G^{n-j,j}_\infty\]
by Theorem \ref{thm:Relative_Spectral_Sequence}, $H^1(\Omega^{n-1}_H,\Omega^n) = 0$. 
\end{proof}

We are now able to prove Theorem \ref{thm:HyperbolicityContainment}.

\HyperbolicityContainment*

\begin{proof}  
Suppose that $\cP_H$ does not contain positive definite matrices. By Theorem \ref{thm:empty_var_iff_hyperbolic_curve}, if $\lambda \in \cP_H$ then $Q_\lambda\vert_H$ has exactly two negative eigenvalues and $n - 2$ positive eigenvalues. Suppose for the sake of a contradiction that $\cP_H \not \subseteq \cP$.  

Let $e \in \cP$ and $a \in \bbR^3$ such that the line $\ell = \{te + (1-t)a \; \vert \; t \in \bbR\} \subseteq \bbR^3$ has $\ell \cap \cV_{\bbR}(g) \cap \cV_\bbR(g_H) = \emptyset$ and the intersections $\ell \cap \cV_{\bbR}(g)$ and $\ell \cap \cV_\bbR(g_H)$ are transverse. This is possible since $\cV_{\bbR}(g) \cap \cV_\bbR(g_H)$ has at most finitely many points. By hyperbolicity of $g$, the polynomial $g(te + (1-t)a)$ has roots $t_1 < t_2 < \ldots < t_{n+1}$. Let $2\leq j \leq n-1$ be such that $t_j < 1 < t_{j+1}$. Then, $g_H(te + (1-t)a)$ has roots $s_1 < s_2 < \ldots < s_n$. By Lemma \ref{lem:zeros_odd} and the Cauchy Interlacing Theorem, there are either three roots $s_{\ell} < s_{\ell+1} < s_{\ell + 2}$ in between $t_{j-2}$ and $t_{j-1}$ or between $t_{j+2}$ and $t_{j + 3}$ or that there are two roots $t_j < s_\ell <  s_{\ell+1}<t_{j+1},$ where $t_{j-2} = -\infty$ if $j-2 = 0$ and $t_{j+3} = \infty$ if $j + 3 = n+2$. By the hypothesis that $\cP_H \not \subseteq \cP$, the last case is impossible. 

In the first two cases, the image of $\cP_{H} \cap \bbS^2$ in $\bbP^2$ is contained in the interior of the oval of $\cV_\bbR(g)$ of depth $\lfloor \frac{n+1}{2} \rfloor - 2$ and the exterior of the oval of $\cV_\bbR(g)$ of depth $\lfloor \frac{n+1}{2} \rfloor - 1$. Note that the image of $\Omega^n$ in $\bbP^2$ is the intersection of the interior of the oval of $\cV_\bbR(g)$ of depth $\lfloor \frac{n+1}{2} \rfloor - 1$ and the exterior of the oval of $\cV_\bbR(g)$ of depth $\lfloor \frac{n+1}{2} \rfloor$. 

Now, the images of $\cP_H \cap \bbS^2$ and $\cP \cap \bbS^2$ in $\bbP^2$ are bounded by two disjoint ovals, neither of which contains the other. Let $\sigma$ be a representative of a nontrivial class in $H_1(\Omega_H^{n}) \simeq H^1(\Omega_H^n)$. By the above observation, $\sigma$ can be chosen so that its image in  $\bbP^2$ intersects the image of $\Omega^n$ and therefore gives a nontrivial class in $H_1(\Omega_H^{n-1},\Omega^n)\simeq H^1(\Omega_H^{n-1},\Omega^n)$, contradicting Corollary \ref{cor:restrict_cohomology_cert}. 
\end{proof}

Figure \ref{fig:Hyperplane_Possibilities} shows the two possible containment patterns for $g$ and $g_H$ on the curve given in Example \ref{ex:ch_not_aggs}. 

\begin{figure}[t]
\begin{subfigure}[b]{0.47\textwidth}
    \includegraphics[width = \textwidth]{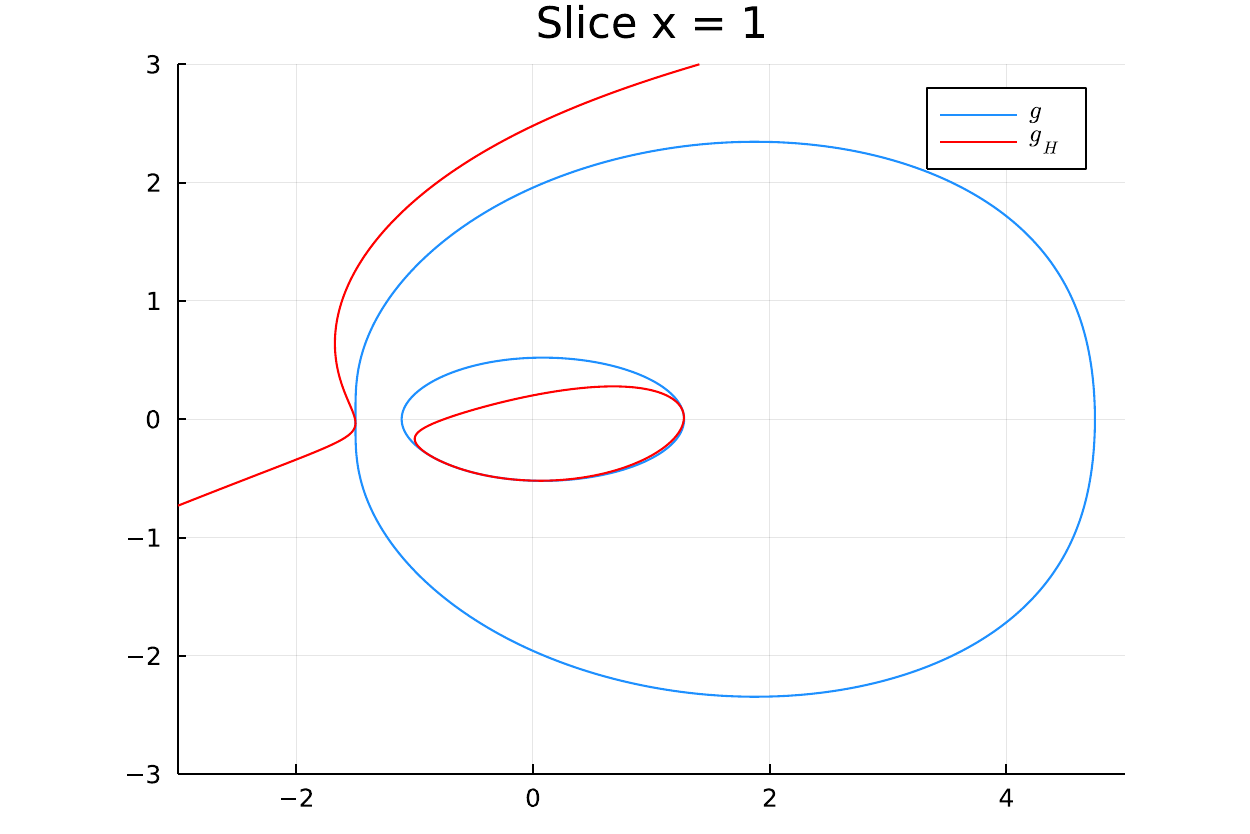}
\end{subfigure}
\begin{subfigure}[b]{0.47\textwidth}
    \includegraphics[width =\textwidth]{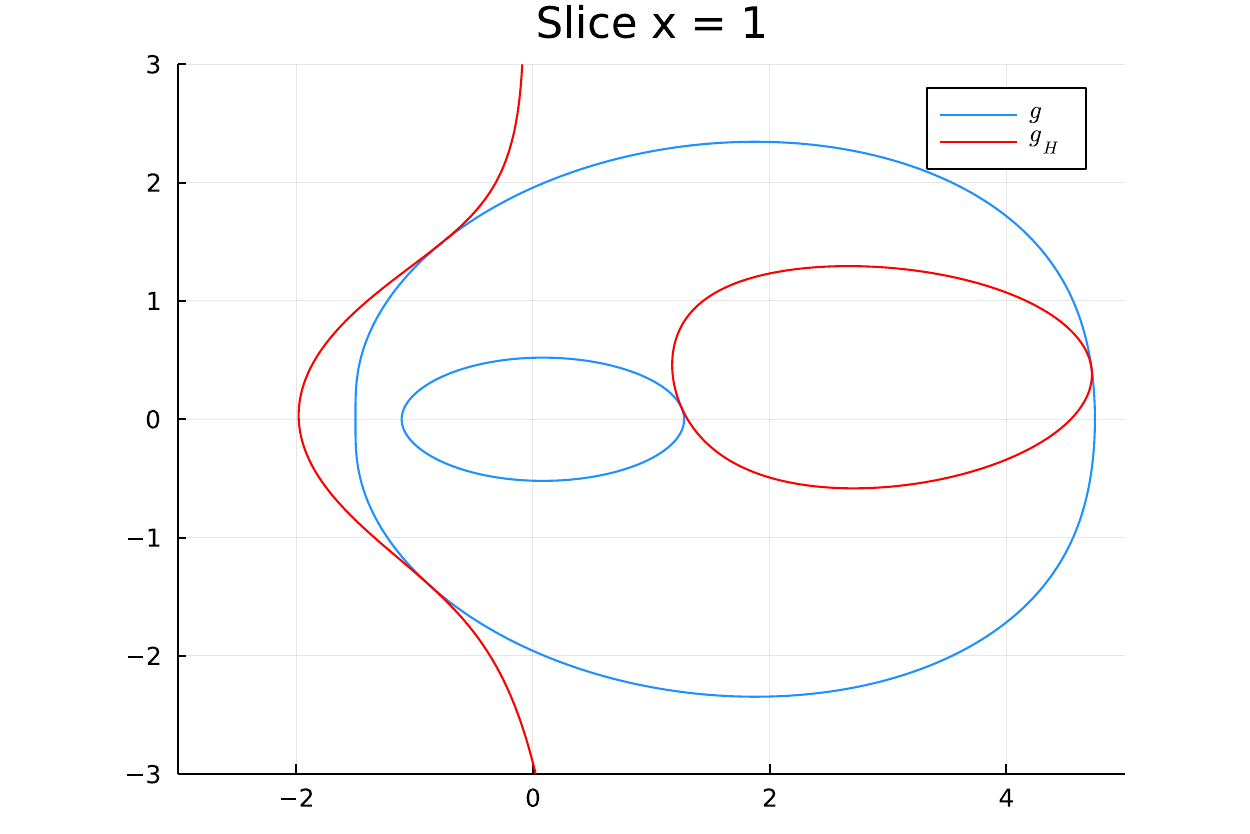}
\end{subfigure}
\caption{The two possible containment patterns for $\cP$ and $\cP_H$ where $g$ is given as in Example \ref{ex:ch_not_aggs}. On the left, $\cP_H \subseteq \cP$ and the emptiness of $\cV_{\bbR}(f^h_1,f^h_2,f^h_3) \cap H$ is certified by $Q_\lambda \vert_H$ having two negative eigenvalues for $\lambda \in \cP_H$ and $H^1(\Omega^{n-1}_H) \not = 0$. On the right, $\cP_H \not \subseteq \cP$ and $Q_\lambda \vert_H \succ 0$ for $\lambda \in \cP_H$.}
\label{fig:Hyperplane_Possibilities}
\end{figure}

We now combine the ideas of Theorem \ref{thm:HyperbolicityContainment} and Proposition \ref{prop:HyperplaneAggs} to obtain a sufficient condition for describing $\cconv(S)$ via aggregations.

\SuffCond*

\begin{proof}
It suffices to show that these hypotheses ensure that \eqref{eq:emptiness_cert} is satisfied. If $\lambda \in \cP$ implies that $Q_\lambda \succ 0$, then the result is \cite[Theorem 2.4]{dey_obtaining_2022}. Otherwise, let $H$ be a hyperplane such that $S^h \cap H = 0$. This implies that the set $X(K,f^h) \cap \bar{H} = \emptyset$ where $K = -\bbR^3_{+}$ is the nonpositive orthant. By Proposition \ref{prop:General_Emptiness},  either there is $\lambda \in \bbR^3_+$ such that $Q_\lambda\vert_H \succ 0$ or $H^1(\Omega^{n-1}_H(K)) \neq 0$. In either case, $g_H$ is hyperbolic if it is smooth. However, since $\Omega^{n-1}(K) \subseteq K^\circ \cap \bbS^2$, a nontrivial group $H^1(\Omega^{n-1}_H(K))$ would imply that $\cP_H \subseteq K^\circ = \bbR^3_+$. On the other hand, by the assumption that $\cV_\bbR(f_1^h,f_2^h,f_3^h) = \emptyset$ and Theorem \ref{thm:HyperbolicityContainment},  $\cP_H \subseteq \cP$. Since $\cP \cap \bbR^3_+ = \emptyset$, this therefore implies that $H^1(\Omega^n(K)) = 0$.

If $g_H$ is not smooth and $H^1(\Omega^{n-1}_H(K)) \neq 0$, then by the Cauchy Interlacing Theorem and since $\cP \cap \bbR^3_+ = 0$, there is $\lambda \in K^\circ$ such that $[\lambda] \in \bbP^2$ lies on the outside of the oval $\cV_\bbR(g)$ of depth $\lfloor \frac{n+1}{2}\rfloor -1$ and such that $Q_\lambda\vert_H$ has at most $n-2$ positive eigenvalues. But then, a nontrivial class in $H^1(\Omega^{n-1}_H(K))$ gives a nontrivial class of  $H^1(\Omega^{n-1}_H(K),\Omega^n(K))$, contradicting Corollary \ref{cor:restrict_cohomology_cert}. 

So, there is $\lambda \in \bbR^3_+$ with $Q_\lambda \vert_H \succeq 0$ and therefore \eqref{eq:emptiness_cert} is satisfied.     
\end{proof}

\begin{remark}
Under the conditions of Theorem \ref{thm:SuffCond}, the assumption that $S$ has no points at infinity can be certified via aggregations. This follows by applying Theorem \ref{thm:HyperbolicityContainment} with $H = \{(x,x_{n+1}) \in \bbR^{n+1} \; \vert \; x_{n+1} = 0\}$, as the hypothesis that $\cP \cap \bbR^3_+ = \emptyset$ implies that the certificate that $S^h \cap H = 0$ cannot come from matrices with two negative eigenvalues.  
\end{remark}

We conclude this section by proving Theorem \ref{thm:ClosedConvexSufficient}.

\ClosedConvexSufficient*

\begin{proof}
The statement that $\cconv(S)$ is given by the intersection of good aggregations when $g$ is hyperbolic is Theorem \ref{thm:SuffCond}. In the $n = 2$ case when $g$ is not hyperbolic, note that condition \eqref{eq:emptiness_cert} is satisfied: A nontrivial class in $H^1(\Omega_H^1(K))$ would give a nontrivial class in $H^1(\Omega_H^1(K),\Omega^2(K))$ in this case, contradicting Corollary \ref{cor:restrict_cohomology_cert}. Therefore by Proposition \ref{prop:General_Emptiness}, there is $\lambda \in \bbR^3_+$ such that $Q_\lambda \vert_H \succ 0$.  

It remains to show that when the spectral curve is hyperbolic, a finite number of good aggregations recovers $\cconv(S)$. For each $y \in \bbR^n \setminus \cconv(S)$, let $\lambda^{(y)} \in \bbR^3$ be a good aggregation which certifies that $y \not \in \cconv(S)$ and set $\Lambda_1 = \{\lambda^{(y)} \; \vert \; y \in \bbR^n \setminus \cconv(S)\}$ to be the set of all such $\lambda^{(y)}$. Now, if $\lambda^{(1)}, \lambda^{(2)}, \ldots, \lambda^{(k)}$ are the unit length generators of the extreme rays of $\Lambda_1$ with $|\supp(\lambda^{(i)})| \leq 2$, then $\bigcap_{i = 1}^k S_{\lambda^{(i)}} = \bigcap_{\lambda \in \Lambda_1} S_\lambda$ by Lemma \ref{lem:UnnecessaryAggstwo} and Propositions \ref{prop:positive_agg_doesnt_contribute} and \ref{prop:General_Emptiness}, as in the proof of Theorem \ref{thm:FiniteAggs}. Note that $k \leq 6$  since each $\lambda^{(i)}$ has $|\supp(\lambda^{(i)})|\leq 2$ and generates an extreme ray of $\Lambda_1$. \end{proof}

\section{Topology of the Set of Permissible Aggregations}\label{sec:Connected_Aggs}

In this section we provide proofs of the statements in Section \ref{sec:Agg_top}. Throughout this section, we assume that $\mathrm{int}(S) \not = \emptyset$. Recall that $\Lambda = \{\lambda \in \bbR^3_+ \; \vert \; Q_\lambda \text{ has exactly one negative eigenvalue}\}$ and that a permissible aggregation $\lambda \in \Lambda$ is a \emph{good aggregation} if $\conv(S) \subseteq S_\lambda$ and a \emph{bad aggregation} otherwise.

\ConnectedGoodAggs*

\begin{proof}
We first show that the set of good aggregations is closed in $\Lambda$. Let $\lambda^{(i)} \to \lambda$ be a sequence of aggregations which converges to $\lambda \in \Lambda$. Recall that each matrix $Q_\lambda$ has block strucutre $Q_\lambda = \begin{bmatrix} A_\lambda & b_\lambda\\ b_\lambda^\top & c_\lambda \end{bmatrix}$, where $A_\lambda \in \bbR^{n \times n}$ defines the homogeneous part of $f_\lambda$. If $\mathrm{int}(S_{\lambda^{(i)}})$ is convex for sufficiently large $i$, then
$A_{\lambda^{(i)}} \succeq 0$ for sufficiently large $i$. Since the cone of positive semidefinite matrices is closed, $A_\lambda \succeq 0$ so that $\mathrm{int}(S_\lambda)$ is convex and therefore $\lambda$ is a good aggregation. Otherwise, suppose that $\mathrm{int}(S_{\lambda^{(i)}})$ has two convex components for large $i$. Let $(\alpha_i,\beta_i)\in \bbR^{n+1}$ be the unit length eigenvector corresponding to the negative eigenvalue of $Q_{\lambda^{(i)}}$ with sign chosen such that $\alpha_i^\top x \leq \beta_i$ for all $x \in S$. This is possible since each $\lambda^{(i)}$ is a good aggregation and the affine hyperplane defined by $\alpha_i^\top x = \beta_i$ separates the two components of $\mathrm{int}(S_{\lambda^{(i)}})$. Now, $\alpha_i \to \alpha$ and $\beta_i \to \beta$ where $(\alpha,\beta)$ is an eigenvector of $Q_\lambda$ corresponding to the negative eigenvalue and $\alpha^\top x \leq \beta$ for all $x \in S$. So, $\lambda$ is a good aggregation.

On the other hand, the set of bad aggregations is also closed in $\Lambda$. Let $\lambda^{(i)}$ be a sequence of aggregations which converges to $\lambda \in \Lambda$. Let $(\alpha_i,\beta_i) \in \bbR^{n+1}$ be the unit length eigenvector of $Q_{\lambda^{(i)}}$ corresponding to the negative eigenvalue of $Q_{\lambda^{(i)}}$ oriented so that $\alpha_i \to \alpha$ and $\beta_i \to \beta$ where $(\alpha,\beta)$ is an eigenvector of $Q_\lambda$ corresponding to the negative eigenvalue. Since each $\lambda^{(i)}$ is a bad aggregation, there are $x,y \in \mathrm{int}(S)$ such that $(\alpha^{(i)})^\top x < \beta^{(i)}$ and $(\alpha^{(i)})^\top y > \beta^{(i)}$ for sufficiently large $i$. This implies that $\alpha^\top x \leq \beta$ and $\alpha^\top y \geq \beta$. Since $x,y \in \mathrm{int}(S)$, this implies the existence of $\tilde{x},\tilde{y} \in \mathrm{int}(S)$ such that $\alpha^\top \tilde{x} < \beta$ and $\alpha^\top \tilde{y} > \beta$. Therefore $\lambda$ is a bad aggregation.

If $\Lambda_1$ is a connected component of $\Lambda$, then we can write $\Lambda_1$ as the disjoint union of good aggregations in $\Lambda_1$ and bad aggregations in $\Lambda_1$. Since $\Lambda_1$ is connected, this implies that $\Lambda_1$ consists entirely of good aggregations or entirely of bad aggregations.
\end{proof}

\begin{remark}\label{rem:PDLC_affine_good}
Note that Theorem \ref{thm:ClosedConvexSufficient} implies that if PDLC is satisfied and $S$ has no points at infinity, then non-intersection of $\cconv(S)$ with an affine hyperplane $H$ is can be certified with a good aggregation. In particular, if $T$ is the intersection of $S_\lambda$ for all good aggregations $\lambda$, then every connected component of $T$ intersects a component of $S$. Indeed, if $T$ had a component which did not intersect $S$, then there would be an affine hyperplane $H$ with $H\cap T \not = \emptyset$ but $H \cap \conv(S) = \emptyset$. 
\end{remark}

We now restrict our attention to the PDLC case. Note that here, unlike in Sections \ref{sec:Finiteness} and \ref{sec:SufficientCondition}, we do not assume that the spectral curve is smooth. The following two results are key ingredients in the proof of \cite[Theorem 2.17]{blekherman_aggregations_2022} and our Theorem \ref{thm:PDLCBound}. The first states that aggregations can be improved by translating in the direction of a positive semidefinite matrix.

\begin{proposition}[{\cite[Proposition 9.1]{blekherman_aggregations_2022}}]\label{prop:Improve_aggs}
    Suppose that there is $\omega \in \bbR^3$ such that $Q_\omega \succ 0$. Set $\Theta = \{\theta \in \bbR^3 \; \vert \; Q_\theta \succeq 0\}$ and $\Lambda_1 = \{\lambda \in \Lambda \; \vert \; \lambda \text{ is a good aggregation and } S_\lambda \not = \bbR^n\}$. If $\lambda \in \Lambda_1$ and $\theta \in \Theta$ are such that $\lambda' = \lambda + \theta \in \bbR^3_+ \setminus \{0\}$, then $\lambda' \in \Lambda_1$ and $S_{\lambda'} \subseteq S_\lambda$.
\end{proposition}

The second states that the intersection over all aggregations is given by the intersection over aggregations with support at most two. 

\begin{proposition}[{\cite[Proposition 9.2]{blekherman_aggregations_2022}}]\label{prop:Aggs_support}
 Suppose that there is $\omega \in \bbR^3$ such that $Q_\omega \succ 0$. Set $\Lambda_1 = \{\lambda \in \Lambda \; \vert \; \lambda \text{ is a good aggregation and } S_\lambda \not = \bbR^n\}$ and $\Lambda_2 = \{\lambda \in \Lambda_1 \; \vert \; |\supp(\lambda)| \leq 2\}$. Then $\bigcap_{\lambda \in \Lambda_1} S_\lambda = \bigcap_{\lambda \in \Lambda_2}S_\lambda$. 
\end{proposition}

\begin{remark}\label{rem:finite_good_aggs_PDLC} Since the convex combination of good aggregations is a good aggregation, it follows from Proposition \ref{prop:Aggs_support} that at most six aggregations are necessary to describe the intersection of all good aggregations when PDLC holds. This was observed in the setting of open inequalities in \cite{blekherman_aggregations_2022}. We will show in Theorem \ref{thm:PDLCBound} that this number can be further reduced to four in the case of closed inequalities when $S$ has no low dimensional components and no points at infinity. 
\end{remark}

\GoodAggsConnectedPDLC*

\begin{proof}

By Theorem \ref{thm:ClosedConvexSufficient}, $\cconv(S)$ can be described as the intersection of finitely many good aggregations. (Compare with \cite[Theorem 2.23]{blekherman_aggregations_2022}, as the $Q_\lambda \not = 0$ for all nonzero $\lambda \in \bbR^3_+$ by the hypothesis that the $Q_i$ are linearly independent and that the $Q_i$ satisfy hidden hyperplane convexity since they satisfy PDLC \cite{blekherman_aggregations_2022}). By Remark \ref{rem:finite_good_aggs_PDLC}, there are $\lambda^{(1)}, \lambda^{(2)}, \ldots, \lambda^{(r)}$, $r \leq 6$ such that $\cconv(S) = \bigcap_{i = 1}^{6}S_{\lambda^{(i)}}$. Set $K = \cone(\lambda^{(1)},\lambda^{(2)},\ldots, \lambda^{(r)})$.

Suppose for the sake of a contradiction that $\Omega^n(K^\circ)$ has two connected components. It then follows by Lemma \ref{lem:Connected_Omega_n} that $\dim H_0(X(K^\circ,f^h)) = 2$. This implies that the set $T$ defined by the intersection of the $S_\lambda$ for all good aggregations $\lambda$ has two connected components. By Remark \ref{rem:PDLC_affine_good}, this implies that $S$ intersects both components. But then, $\conv(S)$ intersects both components, which is a contradiction since $\conv(S)$ is connected. So, the set of good aggregations is connected in $\Lambda$.  \end{proof}

Finally, we develop a certificate that $\conv(S) = \bbR^n$ under the assumption of PDLC. 

\begin{proposition}\label{prop:PDLC_no_invertible_agg_trivial_CH}
If $Q_1,Q_2,Q_3$ satisfy PDLC and $\Omega^{n}(-\bbR^3_+) = \emptyset$ then $\conv(S) = \bbR^n$.
\end{proposition}

\begin{proof}
Suppose for the sake of a contradiction that there is $y \in \bbR^n \setminus \conv(S)$. Since $\{y\}$ is compact, there exist $\alpha \in \bbR^n$ and  $\beta_1,\beta_2 \in \bbR$ with $\beta_1 < \beta_2$ such that $\alpha^\top y = \beta_2$ and $\alpha^\top x < \beta_1$ for all $x \in \cconv(S)$. Let $H \subseteq \bbR^n$ be the hyperplane defined by $\alpha^\top x = \beta_2$ and $H^h = \{(x,x_{n+1}) \; \vert\; \alpha^\top x = \beta_2 x_{n+1}\}$ its homogenization so that $\cconv(S) \cap H = \emptyset$ and $S^h \cap H^h = \{0\}$. Note that $\cV_\bbR(f_1^h, f_2^h,f_3^h) \cap H = \emptyset$ and therefore by proposition \ref{prop:General_Emptiness}, we have that either $\Omega^n_H(-\bbR^3_+) \not = \emptyset$ or $H^1(\Omega^{n-1}_H(-\bbR^3_+)) \not = 0$. However, since the $Q_i$ satisfy PDLC, so do the $Q_i \vert_H$ and therefore  $H^1(\Omega^{n-1}_H(-\bbR^3_+)) = 0$. Since $\Omega^n(-\bbR^3_+) = \emptyset$, it follows that $\Omega^n_H(-\bbR^3_+) = \emptyset$ by the Cauchy Interlacing Theorem. So, no such $y$ can exist and therefore $\cconv(S) = \bbR^n$. 
\end{proof}

We now show that if PDLC is satisfied, $S$ has no points at infinity, and $S = \mathrm{cl}(\mathrm{int}(S))$, then $\cconv(S) = \bigcap_{i = 1}^r S_{\lambda^{(i)}}$ for some $r \leq 4$. Note that \cite[Conjecture 3.2]{blekherman_aggregations_2022} conjectures the existence of a set defined by \emph{strict} quadratic inequalities which cannot be described as the intersection of fewer than six good aggregations. It was observed in \cite[Theorem 2.7]{dey_obtaining_2022} that if $S = \mathrm{cl}(\mathrm{int}(S))$ and $\conv(\mathrm{int}(S)) = \bigcap_{\lambda \in \Lambda_1} \mathrm{int}(S_\lambda)$ for some set of good aggregations $\Lambda_1$, then $\cconv(S) = \bigcap_{\lambda \in \Lambda_1}S_\lambda$.

We use a similar strategy to \cite{blekherman_aggregations_2022} of improving aggregations in the direction of a positive semidefinite combination. Note that if PDLC is satisfied and $S \not = \emptyset$, then the hyperbolicity cone of $g$ which consists of positive definite matrices does not intersect $\bbR^3_+$ away from 0. However, it is posible that the hyperbolicity cone of $g$ which contains negative definite matrices intersects $\bbR^3_+$ and that this hyperbolicity is contained strictly in $\bbR^3_+$. 

We first deal with the case that the cone of negative definite matrices is not strictly contained in $\bbR^3_+$ and employ the strategy of Proposition \ref{prop:Improve_aggs} to improve aggregations along the direction of a positive semidefinite matrix. 

\begin{proposition}\label{prop:drag_away_from_edge}
Spupose that $Q_1,Q_2,Q_3$ satisfy PDLC, $\mathrm{int}(S) \not = \emptyset$, and $S$ has no points at infinity. Assume that the hyperbolicity cone of $g$ which contains negative definite matrices is not contained in $\bbR^3_+$. Then, there are $\lambda^{(1)},\lambda^{(2)}, \ldots, \lambda^{(r)}$ such that $r\leq 4$ and $\cconv(S) = \bigcap_{i = 1}^r S_{\lambda^{(i)}}$. 
\end{proposition}

\begin{proof}
If the hyperbolicity cone of $g$ which contains negative definite matrices is not contained in $\bbR^3_+$, then there is $\theta \in \bbR^3$ such that $Q_\theta \succeq 0$ and at least one component of $\theta$ is strictly positive. Suppose without loss of generality that $\theta_3 >0$.  Note also that at least one of $\theta_1,\theta_2 <0$ since otherwise $\theta \in \bbR^3_+$ and therefore $\mathrm{int}(S) = \emptyset$. 

Now, if $\lambda =  \lambda_1e_1 + \lambda_2e_2 \in \Lambda$ is an aggregation for $\lambda_1,\lambda_2 > 0$, we have that $\lambda + \epsilon \theta \in \bbR^3_+\setminus \{0\}$, where 

\[\epsilon = \min \left\{\frac{-\lambda_i}{\theta_i} \; \middle\vert \; \theta_i < 0\right\} > 0.\]
So, by Proposition \ref{prop:Improve_aggs}, $S_{\lambda + \epsilon \theta}$ is an aggregation with $S_{\lambda + \epsilon \theta} \subseteq S_\lambda$ and $\lambda + \epsilon \theta \in  \cone(e_1,e_3) \cup \cone(e_2,e_3)$. In particular, this implies that no aggregations in $\cone(e_1,e_2)$ are necessary for the description of $\conv(S)$. By Proposition \ref{prop:Aggs_support} and the fact that all aggregations in $\cone(e_i,e_j)$ can be described using at most two aggregations, this in turn implies that $\cconv(S)$ can be described as the intersection of at most four aggregations. 
\end{proof}

If $n = 1$ or $2$, then we are able to eliminate the possibility of six aggregations being necessary using the structure of the curve $\cV_{\bbR}(g)$. As in the statement of Proposition \ref{prop:Aggs_support}, set $$\Lambda_1 = \{\lambda \in \Lambda \; \vert \; \lambda \text{ is a good aggregtaion and } S_\lambda \neq \bbR^n\},$$ and $\Lambda_2 = \{\lambda \in \Lambda_1 \; \vert\; |\supp \lambda| \leq 2\}$. Set $\Lambda_3$ to be the set of unit length generators of the extreme rays of $\Lambda_2$. 

\begin{proposition}\label{Prop:low_Degree_PDLC_Bound}
If $n = 1$ or $n = 2$, $Q_1,Q_2,Q_3$ satisfy PDLC, $\mathrm{int}(S) \not = \emptyset$, $S = \mathrm{cl}(\mathrm{int}(S))$, $S$ has no points at infinity, and the hyperbolicity cone of $g$ which contains negative definite matrices is contained in $\bbR^3_+$, then there are $\lambda^{(1)},\lambda^{(2)}, \ldots, \lambda^{(r)} \in \Lambda_2$ such that $r \leq 4$ and $\cconv(S) = \bigcap_{\lambda \in \Lambda_2} S_\lambda = \bigcap_{i = 1}^rS_{\lambda^{(i)}}$. 
\end{proposition}

\begin{proof}

It follows from Proposition \ref{prop:Aggs_support} that $\cconv(S) = \bigcap_{\lambda \in \Lambda_2} S_\lambda$. Since the convex combination of good aggregations is a good aggregation if it is permissible, it follows that $\cconv(S) = \bigcap_{\lambda \in \Lambda_3} S_\lambda$ and therefore it suffices to bound $|\Lambda_3|$. If $\lambda \in \Lambda_3$, then either $\lambda$ is a standard basis vector or $g(\lambda) = 0$. Note that if $\lambda \in \Lambda_3$, then $Q_\lambda$ has exactly one negative eigenvalue since if $Q_\lambda \succeq 0$ then $\mathrm{int}(S_\lambda) = \emptyset$, a contradiction with the hypothesis that $\mathrm{int}(S)\not = \emptyset$ since $\mathrm{int}(S) \subseteq \mathrm{int}(S_\lambda)$. Note also that the set of good aggregations is connected in $\Lambda$ by Propositions \ref{prop:ConnectedGoodAggs} and \ref{prop:GoodAggsConnectedPDLC}.  

In the case $n = 1$, note that for $\lambda \in \bbR^3_+$, either $Q_\lambda \prec 0$ or $\lambda \in \Lambda$. Since the cone of negative definite matrices is contained in $\bbR^3_+$, it follows that every nonzero $\lambda \in \bbR^3_+$ with $|\supp(\lambda)| \leq 2$ has exactly one negative eigenvalue. In particular, we have that $\Lambda$ is connected, and therefore the set of good aggregations is also connected. It then follows that $\Lambda_3 = \{e_1,e_2,e_3\}$.

In the case $n = 2$, if $\lambda \in \Lambda_3$ is not a standard basis vector, then $\lambda$ lies on the non-oval component of $\cV_\bbR(g)$. Note that the line through $e_j$ and $e_k$ can only intersect this component of the spectral curve once.  So, if a face $\cone(e_j,e_k)$ of $\bbR^3_+$ contains two elements of $\Lambda_3$, it follows that the elements are $e_j,e_k$ or that exactly one of $e_j,e_k$ is an element of $\Lambda_3$. So, up to reordering of the $Q_i$, the three possibilities for $\Lambda_3$ are 

\[\{e_1,e_2,e_3\}, \quad \{e_1, \lambda^{(2)}, \lambda^{(3)}\}, \text{ and} \quad \{e_2,e_3,\lambda^{(2)},\lambda^{(3)}\},\]
where $\lambda^{(i)} \in \cone(e_1,e_i)$ is not a standard basis vector. 

In all cases, $|\Lambda_3| \leq 4$. 
\end{proof}

We now address the case that $n\geq 3$ and the hyperbolicity cone of $g$ which contains negative definite matrices is strictly contained in $\bbR^3_+$. To do so, we need the following lemma which controls the eigenvalues of matrices which are convex combinations of an element of $\Lambda_3$ and a negative definite matrix.

\begin{lemma}\label{lem:neg_def_cvx_comb}
Suppose that the hyperbolicity cone of $g$ which contains negative definite matrices is strictly contained in $\bbR^3_+$. Suppose that $\lambda \in \Lambda_3$ has $g(\lambda) = 0$. Let $\omega \in \bbR^3$ be such that $Q_\omega \prec 0$. Then, for any $t \in (0,1)$, the matrix $Q_{t\lambda + (1-t)\omega}$ has at most $n-1$ positive eigenvalues.    
\end{lemma}

\begin{proof}
Consider the univariate polynomial $g(t\lambda + (1-t)\omega)$. Suppose for the sake of a contradiction that there is $t^* \in (0,1)$ such that $Q_{t^*\lambda + (1-t^*)\omega}$ has $n$ positive eigenvalues. Then, since $g(\lambda) = 0$ and $Q_\omega$ has $n+1$ negative eigenvalues, there are $n+1$ roots of $g(t\lambda + (1-t)\omega)$ on the interval $(0,1]$, counted with multiplicity. However, since $\lambda$ is not in the hyperbolicity cone of $g$ which includes $\omega$, it cannot be the case that $g(t\lambda + (1-t)\omega)$ has all positive roots. So, there is a negative root ($g(0) = \det(Q_\omega) \not = 0$). But then, the degree $n+1$ polynomial $g(t\lambda + (1-t)\omega)$ has $n+2$ roots, a contradiction with the fact that $g$ is nonconstant.  
\end{proof}

Next, we show that the hyperbolicity cone of $g$ which contains negative definite matrices does not intersect $\cone(\Lambda_3)\setminus\{0\}$. 

\begin{lemma}\label{lem:neg_def_cone_lambda_prime}
Suppose $n \geq 3$ and that the hyperbolicity cone of $g$ which contains negative definite matrices is strictly contained in $\bbR^3_+$. Let $\omega \in \cone(\Lambda_3)$ be nonzero. Then, $Q_\omega$ has at least one nonnegative eigenvalue.
\end{lemma}

\begin{proof}

 By Carath\'eodory's Theorem, $\omega$ is a conical combination of $\lambda^{(1)},\lambda^{(2)}, \lambda^{(3)} \in \Lambda_3$. Let $V_1,V_2,V_3$ be the $n$-dimensional subspaces of $\bbR^{n+1}$ so that $v^\top Q_{\lambda^{(i)}}v \geq 0$ for $v \in V_i$. Then, $V = \cap_{i = 1}^{3}V_i$ has codimension at most three in $\bbR^{n+1}$ and therefore $\dim (V) \geq 1$ since $n \geq 3$. Moreover, for $v \in V$ and $t_1,t_2,t_3 \geq 0$, we see that 

\[v^\top \left(\sum_{i =1}^{3}t_iQ_{\lambda^{(i)}}\right) v = \sum_{i = 1}^{3}t_i v^\top Q_{\lambda^{(i)}}v \geq 0.\]
So, if $\omega \in \cone(\lambda^{(1)},\lambda^{(2)},\lambda^{(3)}),$ then $Q_{\omega}$ has at least one nonnegative eigenvalue.  
\end{proof}

Using Lemmas \ref{lem:neg_def_cvx_comb} and \ref{lem:neg_def_cone_lambda_prime}, we show that it is not possible for two aggregations from each two dimensional face of $\bbR^3_+$ to be necessary for the description of $\conv(S)$.

\begin{proposition}\label{prop:Neg_def_contained_four_aggs}
Suppose that $n\geq 3$, and that $Q_1,Q_2,Q_3$ satisfy PDLC, $S \not = \emptyset$, $S$ has no points at infinity, and $S = \mathrm{cl}(\mathrm{int}(S))$. Assume that the hyperbolicity cone which consists of negative definite matrices is completely contained in $\bbR^3_+$. Then, $\cconv(S)$ can be described using at most two aggregations. 
\end{proposition}

\begin{proof}

Let $\lambda^{(1)},\lambda^{(2)},\ldots, \lambda^{(r)}$ be good aggregations such that $\conv(S) = \bigcap_{i = 1}^{r}S_{\lambda^{(i)}}$. Set $K = \cone(\lambda^{(1)},\lambda^{(2)}, \ldots, \lambda^{(r)})$. Note that at most two of the $\lambda^{(i)}$ can lie on a given facet $\cone(e_j,e_k)$ of $\bbR^3_+$. By Propositions \ref{prop:GoodAggsConnectedPDLC} and \ref{prop:PDLC_no_invertible_agg_trivial_CH}, we have that $\Omega^n:=\Omega^n(K^\circ)$ is nonempty and connected. 

We first show that the connected component of $\Lambda$ which contains $\Omega^n$ contains exactly two of the $\lambda^{(i)}$. Without loss of generality relabel so that these are $\lambda^{(1)},\lambda^{(2)}$.  Suppose for the sake of a contradiction that there is $\lambda^{(3)}$ in the same connected component of $\Lambda$. Given $\omega$ such that $Q_\omega \prec 0$, set $L_i = \{t\omega + (1-t)\lambda^{(i)}\;\vert\; t \in [0,1]\}$ and $\hat{L_i} = \left\{\frac{\ell}{\|\ell\|}\; \middle \vert \; \ell \in L_i\right\}$. By Lemma \ref{lem:neg_def_cvx_comb}, it follows that $Q_\mu$ has at most $n-1$ positive eigenvalues for any $\mu \in \hat{L_i}$ and any $i \in [3]$. So, $\hat{L_i} \cap \Omega^n = \emptyset$ for each $i \in [3]$ and therefore $\Omega^n$ has at least two components, the desired contradiction. 

Note that it suffices to show that for any affine hyperplane $H\subseteq \bbR^{n+1}$, the set $\left(\bigcap_{i=1}^r S_{\lambda^{(i)}}\right) \cap H = \emptyset$ if and only if $H \cap S_{\lambda^{(1)}} \cap S_{\lambda^{(2)}} = \emptyset$. Indeed, since $\cconv(S) = \bigcap_{i = 1}^r S_{\lambda^{(i)}}$, if $y \not \in \bigcap_{i = 1}^{r}S_{\lambda^{(i)}}$ then there is an affine hyperplane $H$ such that $y \in H$ and $H \cap \left(\bigcap_{i = 1}^{r} S_{\lambda^{(i)}}\right) = \emptyset$ and therefore $y \not \in S_{\lambda^{(1)}} \cap S_{\lambda^{(2)}}$. On the other hand, since $S_{\lambda^{(i)}}$ is a good aggregation for each $i \in [r]$, it follows that $\cconv(S) \subseteq S_{\lambda^{(1)}} \cap S_{\lambda^{(2)}}$.  

Clearly $\left(\bigcap_{i = 1}^r S_{\lambda^{(i)}}\right) \cap H  = \emptyset$ if $S_{\lambda^{(1)}} \cap S_{\lambda^{(2)}} \cap H = \emptyset$ for any affine hyperplane $H$. For the converse statement, suppose that $H$ is an affine hyperplane such that $\left(\bigcap_{i =1}^r S_{\lambda^{(i)}}\right) \cap H = \emptyset$. Then, by Proposition \ref{prop:General_Emptiness},  $\Omega^n_H(K^\circ) \not = \emptyset$, as $H^1(\Omega_H^{n-1}(K^\circ)) = 0$ since PDLC is satisfied. By the Cauchy Interlacing Theorem, $\Omega^n_H(K^\circ) \subseteq \Omega^n$. Moreover, $\Omega^n_H(K^\circ) \cap \cone(\lambda^{(1)},\lambda^{(2)}) \not = \emptyset$ since $g_H$ interlaces $g$ and therefore the hyperbolicity cone of $g_H$ containing positive definite matrices cannot be completely contained in $\bbR^3_+$. Since $\Omega^n_H(K^\circ) \cap \cone(\lambda^{(1)},\lambda^{(2)}) \not = \emptyset$, it follows that $S_{\lambda^{(1)}} \cap S_{\lambda^{(2)}} \cap H = \emptyset$. 
\end{proof}

Combining the results of Propositions \ref{prop:drag_away_from_edge}, \ref{Prop:low_Degree_PDLC_Bound}, and \ref{prop:Neg_def_contained_four_aggs} then implies Theorem \ref{thm:PDLCBound}.

\PDLCBound* 

\begin{proof}

If the $Q_i$ satisfy PDLC, then $g$ is hyperbolic with a hyperbolicity cone which contains negative definite matrices. Proposition \ref{prop:drag_away_from_edge} proves the theorem in the case that this hyperbolicity cone is not entirely contained in $\bbR^3_+$. In the case where this hyperbolicity cone is contained in $\bbR^3_+$ and $n =1,2$, the statement is Proposition \ref{Prop:low_Degree_PDLC_Bound}. The remaining case where this hyperbolicity cone is contained in $\bbR^3_+$ and $n \geq 3$ is proved in Proposition \ref{prop:Neg_def_contained_four_aggs}. 
\end{proof}

\printbibliography

\end{document}